\documentclass[reqno, 11pt]{amsart}
\usepackage{times}
\usepackage[colorlinks]{hyperref}
\usepackage{url}
\newtheorem{theorem}{Theorem}[section]
\newtheorem{lemma}[theorem]{Lemma}
\newtheorem{proposition}[theorem]{Proposition}
\newtheorem{corollary}[theorem]{Corollary}

\theoremstyle{definition}
\newtheorem{definition}[theorem]{Definition}
\newtheorem{ex}[theorem]{Example}

\newtheorem{remark}[theorem]{Remark}
\newtheorem{algo}[theorem]{Algorithm}

\theoremstyle{remark}

\numberwithin{equation}{section}

\usepackage{bbm}
\usepackage{url}
\usepackage{euscript}
\usepackage{pb-diagram}
\usepackage{lamsarrow}
\usepackage{pb-lams}
\usepackage{amsmath}
\usepackage{amsthm}

\usepackage{graphicx}
\usepackage{epstopdf}

\usepackage{amssymb}
\usepackage{pifont}
\usepackage{amsbsy}

\newskip\aline \newskip\halfaline
\def\skipaline{\vskip\aline}

\def\qedbox{$\rlap{$\sqcap$}\sqcup$}
\def\qed{\nobreak\hfill\penalty250 \hbox{}\nobreak\hfill\qedbox\skipaline}

\newcommand\bR{{\mathbb R}}

\newcommand\bZ{{\mathbb Z}}

\newcommand{\bse}{{\boldsymbol{e}}}

\newcommand{\bsS}{\boldsymbol{S}}

\newcommand{\be}{\boldsymbol{e}}

\newcommand{\bh}{\boldsymbol{h}}

\newcommand{\bn}{{\boldsymbol{n}}}
\newcommand{\bp}{\boldsymbol{p}}
\newcommand{\bq}{\boldsymbol{q}}

\newcommand{\bv}{{\boldsymbol{v}}}

\newcommand{\bsC}{\boldsymbol{C}}

\newcommand{\bsI}{\boldsymbol{I}}
\newcommand{\bsJ}{\boldsymbol{J}}
\newcommand{\bsK}{\boldsymbol{K}}

\newcommand{\bsN}{\boldsymbol{N}}

\newcommand{\bsi}{\boldsymbol{\sigma}}

\newcommand{\si}{{\sigma}}

\newcommand{\ve}{{\varepsilon}}
\newcommand{\ep}{{\varepsilon}}

\newcommand{\vfi}{{\varphi}}

\newcommand{\eC}{\EuScript{C}}

\newcommand{\eE}{\EuScript{E}}
\newcommand{\eF}{\EuScript{F}}
\newcommand{\eG}{\EuScript{G}}

\newcommand{\eI}{\EuScript{I}}
\newcommand{\eJ}{\EuScript{J}}

\newcommand{\eL}{\EuScript{L}}

\newcommand{\eN}{\EuScript{N}}

\newcommand{\eP}{{\EuScript{P}}}
\newcommand{\eQ}{\EuScript{Q}}
\newcommand{\eR}{\EuScript{R}}
\newcommand{\eS}{\EuScript{S}}

\newcommand{\eV}{\EuScript{V}}

\newcommand{\ra}{\rightarrow}

\newcommand{\Llra}{\Longleftrightarrow}

\newcommand{\lan}{\langle}
\newcommand{\ran}{\rangle}

\def\inpr{\mathbin{\hbox to 6pt{\vrule height0.4pt width5pt depth0pt \kern-.4pt \vrule height6pt width0.4pt depth0pt\hss}}}

\newcommand{\pa}{\partial}

\newcommand{\pato}{\partial_{\mathrm{top}}}

\DeclareMathOperator{\sign}{\mathsf{sign}}

\DeclareMathOperator{\dist}{dist}

 \DeclareMathOperator{\Hom}{Hom}

\DeclareMathOperator{\cl}{\boldsymbol{cl}}

\DeclareMathOperator{\stack}{\textsf{stack}}
\DeclareMathOperator{\jump}{\textsf{jump}}
\DeclareMathOperator{\polygon}{\textsf{polygon}}

\begin{document}

\title{Pixelations of planar semialgebraic sets and shape recognition}

\author{Liviu I. Nicolaescu}

\address{Department of Mathematics, University of Notre Dame, Notre Dame, IN 46556-4618.}
\email{nicolaescu.1@nd.edu}
\urladdr{http://www.nd.edu/~lnicolae/}

\author{Brandon Rowekamp}

\address{Department of Mathematics \& Statistics, Minnesota State University, Mankato, 273 Wissink Hall, Mankato, MN 56001.}

\email{brandon.rowekamp@mnsu.edu}

\thanks{The first author  was partially supported by the NSF grant DMS-1005745}
\keywords{semialgebraic sets, pixelations,  normal cycle,   total curvature, Morse theory}
\subjclass[2000]{53A04, 53C65, 58A25}

\date{\today}
\begin{abstract} We   describe an algorithm that associates  to each  positive real number $\ve$  and each finite collection $C_\ve$ of planar pixels of size $\ve$   a planar  piecewise linear  set $S_\ve$  with the following additional property: if $C_\ve$ is the collection of pixels of size $\ve$ that touch a given compact semialgebraic   set $S$, then  the normal cycle of  $S_\ve$ converges  in the sense of currents  to the  normal cycle of   $S$.  In particular, in the limit we can recover   the homotopy type  of $S$ and its geometric invariants such as area, perimeter and curvature measures. At its core,  this algorithm is a    discretization  of stratified Morse theory.

\end{abstract}

\maketitle

\tableofcontents

\section*{Introduction}
\label{s: 0}
\setcounter{equation}{0}

This  paper is a natural  sequel of the  investigation begun by the second author in his dissertation \cite{Row}.   To formulate the  main  problem discussed in \cite{Row} and in this paper we need to introduce a bit of terminology.   

For $\ep>0$ we define an $\ep$-\emph{pixel} to be a  square of the form 
\[
[(m-1)\ep,m\ep]\times [(n-1)\ep,n\ep]\subset \bR^2, \;\;m,n\in\bZ.
\]
The  number $\ep$ is called the \emph{resolution}.  A \emph{pixelation} is a union   of finitely many    pixels.  A \emph{column}  of the pixelation  is   the intersection of  the pixelation with a vertical strip of the form $\{(m-1)\ep< x < m\ep\}$. The   $\ep$-pixelation of a set $S\subset \bR^2$ is the union of all the pixels that touch $S$. We denote it  by $P_\ep(S)$.     The pixelation  $P_\ve(S)$  can be viewed as a discretization of   the tube 
\[
T_\ve(S)=\Bigl\{ \bp\in \bR^2;\;\; \min_{\bq\in S}\|\bp-\bq\|_\infty\leq \frac{\ve}{2}\,\Bigr\},\;\;\|(x,y)\|_\infty:=\max\{|x|,|y|\,\}.
\]
More precisely, if $\Lambda_\ve$ denotes the (affine) lattice consisting of the centers of all the $\ve$-pixels, then the  set of centers of the pixels in $P_\ve(S)$ is $T_\ve(S)\cap \Lambda_\ve$.

\medskip

\noindent {\bf  The Main Problem.} Produce an algorithm that associates to $P_\ep(S)$ a $PL$-set $S_\ep$ which  approximates $S$ very well as $\ep\searrow 0$. More precisely, for $\ep>0$ sufficiently small, the approximation   $S_\ep$  must   have the same homotopy type  as $S$ and  the curvature features  of $S_\ep$  must  closely resemble those of $S$.   

\medskip

We will be more more accurate about what we mean  by curvature features.  For now it helps to think that  $S$ is a $C^2$-curve in the plane describing the contour  of a planar shape.   Then  the sharp angles of the  $PL$-set    $S_\ve$ should be located near the points of high curvature of the  contour $S$.  Similarly, the concavities  and convexities of $S_\ve$ should closely track those of $S$.  Thus,  if $S_\ve$  is known to approximate   a contour from a finite list $\eL$   of contours, then for $\ve>0$ sufficiently small we should be able to recognize which   contour    in $\eL$  corresponds to $S_\ve$. 

The only input we have for the $PL$-approximation  consists of  a rather blurry information about $S$, namely the pixelation $P_\ve(S)$. This pixelation is also a $PL$-set, and one could reasonably ask,  why not use $P_\ep(S)$ as the sought for $PL$-approximation. One geometric obstruction is immediately visible: the pixelation $P_\ve(S)$  very jagged and there is no hope that its curvature    properties  are similar to those of $S$. In fact there is a more insidious reason why  the pixelation is a poor approximation  for $S$.

\begin{figure}[ht]
\centering{\includegraphics[height=1.5in,width=1.5in]{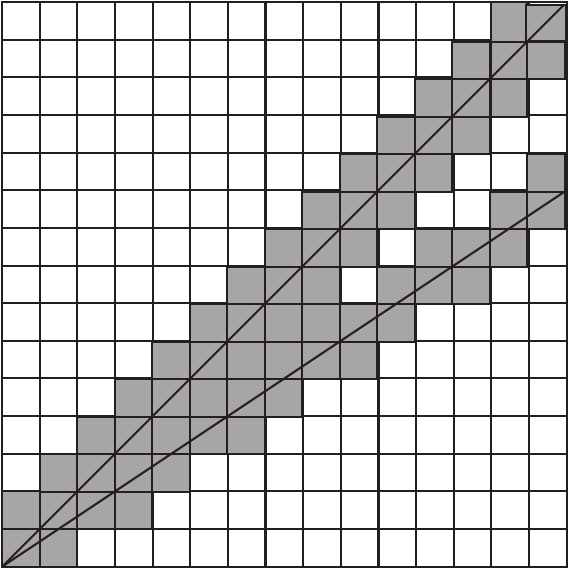}}
\caption{\sl The pixelation of the angle $A\left(1,\frac{2}{3}\right)$ contains two holes.}
\label{fig: 2hole0}
\end{figure}

Consider the pixelation of an angle $A\bigl(1,\frac{2}{3}\,\bigr)$ with  vertex at the origin whose edges  have slopes $\frac{2}{3}$ and $1$.   Figure \ref{fig: 2hole0} shows that this pixelation is not contractible and in fact its first Betti number is $2$.  These  two ``holes''   won't disappear at any resolution because all the pixelations of this angle are rescalings of each other.   

Things can get a lot worse. For example, if $n$ is a positive integer and  $S_n$ is the union of the two line segments connecting the origin to the points $(n,2n+1)$  and $(1,2)$, then for any  $\ve>0$  sufficiently small we have $b_1(\,P_\ve(S_n)\,)=2n$, while obviously $b_1(S_n)=0$.

In \cite{Row} the second  author  solved the  Main Problem in the special case when $S$ itself is  a $PL$-set. The resulting  algorithm is based on     two key principles  inspired by   Morse theory.

\medskip 

\noindent  {\bf  Principle 1.} Suppose that   $S$ is  the graph of a continuous piecewise $C^2$-function  $f:[a,b]\to\bR$,  and the second order derivatives of $f$ are bounded.  We fix a  function $\si:(0,\infty)\ra \bZ_{>0}$, called the \emph{spread}, such that 
\begin{equation*}
\lim_{\ep\searrow 0}\ep\si(\ep)=0\;\;\mbox{and}\;\; \lim_{\ep\searrow 0}\ep\si(\ep)^2=\infty
\tag{$\bsi$}
\label{tag: bsi}
\end{equation*}
For any $\ve>0$ every column of the pixelation  $P_\ve(S)$ is connected. For each $\ep>0$ we obtain by linear interpolation a $PL$ function  $\tau_\ep$ (resp. $\beta_\ep$)  whose  graph is  produced by connecting with straight line segments  the centers of the top (resp. bottom)  pixels of every   $\si(\ep)$-th column of the pixelation of the graph of $f$; see Figure \ref{fig: interpolate} where $\si=3$.

\begin{figure}[ht]
\centering{\includegraphics[height=2.2in,width=2.2in]{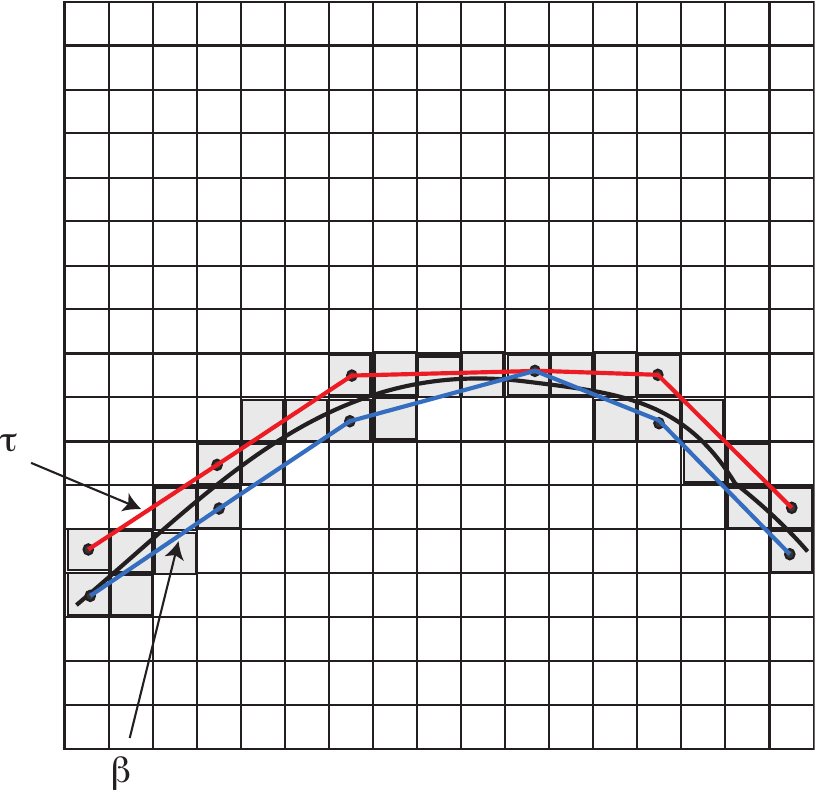}}
\caption{\sl Linear interpolations with spread $\si=3$.}
\label{fig: interpolate}
\end{figure}

The result of the algorithm   is the $PL$-region $S_\ep$ between the graphs of $\beta_\ep$ and $\tau_\ep$. This is a very narrow two dimensional $PL$ set very close to the graph of $f$.  Moreover, the condition (\ref{tag: bsi}) gurantees that the curvature of $S_\ep$ resembles that of  $S$.

An identical strategy works   when $S$ is a set of the form
\[
S=\bigl\{ (x,y)\in\bR^2;\;\;x\in [a,b],\;; \beta(x)\leq y\leq \tau(x)\,\bigr\},
\]
where $\beta,\tau: [a,b]\ra \bR$ are H\"{o}lder  continuous, piecewise $C^2$-functions   such that $\beta(x) \leq \tau(x)$, $\forall x\in [a,b]$.

We will refer to  these  two types of sets as \emph{elementary}.   Thus, the Main Problem has a solution for elementary sets.

\medskip

\noindent {\bf Principle 2.} Suppose that $S\subset \bR^2$ is a \emph{generic} $PL$-set, i.e., its $1$-dimensional skeleton does not contain vertical segments. Consider the linear map $h(x,y)=x$.        The Morse theoretic properties  of the restriction of $h$ to $P_\ep(S)$  closely mimic the Morse theoretic properties  of the restriction of $h$ to $S$ if $\ep$ is sufficiently small.   Here are the details.

 For $x_0\in \bR$ denote by $\bn_{S}(x_0)$ the number of connected components of the intersection of $S$ with the  vertical line $\{x=x_0\}$ and denote by $\eJ_S$ the set of  discontinuities of the function $x\mapsto \bn_S(x)$.  Then $\eJ_S$ is a finite subset of $\bR$ and  there exists $\gamma>0$ such that for any  $r\in (0,\gamma)$  the    set $S'_r$ obtained from $S$ by removing the vertical strips $\{|x-j|<r\}$, $j\in \eJ_S$, is a disjoint union  of  elementary regions.

The set $\eJ_S$ is difficult to determine from a  pixelation, but one can algorithmically  produce  a very small region  containing it.   Here  is   roughly the strategy.

  For $\ep>0$ and $x_0\in \bR$ we denote by $\bn_{S,\ep}(x_0)$ the number of connected components of the intersection of the vertical line    $\{x=x_0\}$ with the pixelation $P_\ep(S)$.   We denote by $\eJ_{S,\ep}$ the set of discontinuities of the   function $x\mapsto \bn_{S,\ep}(x)$.   The set $\eJ_{S,\ep}$ is finite and one can prove  the following remarkable \emph{robustness result}. 

\smallskip

\noindent ($\boldsymbol{R}_0$)   There exist $\hbar,\nu_0>0$  depending only on $S$, such that for $\ep<\hbar$ we have  
\[
\dist(\eJ_S,\eJ_{S,\ep})< \nu_0 \ep.
\]
Above,  $\dist$ refers to the Hausdorff  distance.  We define the \emph{noise region} to be   the set
\[
\eN_\ep:=\bigl\{ x\in \bR;\;\;\dist(x,\eJ_{S,\ep})\leq 2\nu\ep\,\bigr\}.
\]
For $\ep$ sufficiently small, the noise region is a finite union of disjoint compact  intervals 
\[
I_j(\ep), \;\;j=1,\dotsc, N:=\# \eJ_S,
\]
 called  \emph{noise} intervals.   

We denote  by $P'_\ep(S)$ the closure of the set obtained  from $P_\ep(S)$ by removing the vertical  strips $\{x\in I_j(\ep)\}$, $j=1,\dotsc, N$. Each of the connected components of $P'_\ep(S)$ is the pixelation of an elementary set and as such it can  be  $PL$-approximated  using  {\bf Principle 1}.    

The approximation above the noise intervals, i.e., the  intersection of $P_\ep(S)$ with the above vertical strips is  rather coarse.  Every component   of such a region is approximated by the smallest rectangle     that contains it. Here by rectangle we mean   a region of the  form $[a,b]\times [c,d]$, $a\leq b$, $c\leq d$.

\medskip

It turns out that  the approximation $S_\ep$ of $S$ obtained in this fashion from $P_\ep(S)$ is very good in the  following sense:  \emph{the normal cycle  of $S_\ep$ converges in the sense of currents   to the normal cycle of $S$.}     For a  nice introduction to the subject of normal cycles we refer to \cite{Mor}.  A brief description of this concept can also be  found on page \pageref{p: normal} of this paper.

The goal of this paper  is to extend the above program to the more general case of compact, semialgebraic subsets  of $\bR^2$.  While  {\bf  Principle 1} extends with only  little extra effort to the semi-algebraic case,   {\bf  Principle  2} requires  a   more delicate analysis.   This requires that $S$ be a generic  semialgebraic  set in the sense that the restriction  to $S$ of  the linear function $h(x,y)=x$ be a \emph{stratified Morse function} in the sense of Goresky-MacPherson; see  \cite{GM, Pig} or Section \ref{s: 1}.      In this case  the set $\eJ_S$ can be alternatively described as the set of critical values of $h|_S$ corresponding to critical points whose Morse  data in the sense of Goresky-MacPherson \cite{GM} are homotopically nontrivial.

  We know that the stratified Morse function $h|_S$  is stable, \cite{Pig}.    Remarkably the function $h$ is also robust: some   of the topological features of $h|_S$  are preserved if we slightly alter $S$ in a rather irregular way, by replacing it with one of its pixelations.  More precisely  we have the following    counterpart  of  ($\boldsymbol{R}_0$).

\smallskip

\noindent ($\boldsymbol{R}$)   Suppose that $S\subset \bR^2$ is a generic, compact semialgebraic set. Then there exist $\hbar>0$, $\nu_0>0$ and $\kappa_0\in (0,1]$, depending only on $S$ such that for $\ep<\hbar$ we have  
\[
\dist(\eJ_S,\eJ_{S,\ep})< \nu_0 \ep^{\kappa_0}.
\]

\smallskip

The main difference between ($\boldsymbol{R}$) and ($\boldsymbol{R}_0$) is the presence of the exponent $\kappa_0\in (0,1]$.    This   exponent   takes into account the possibility that  the $1$-dimensional skeleton of $S$ may have  cusps such as $|y|^p=x^q$, $x\geq 0$, $\alpha=\frac{q}{p}\geq 1$.  The  higher the orders of contact $\alpha$ of such cusps, the lower the exponent $\kappa_0$. In fact $\kappa_0 \leq \frac{1}{\alpha}$ for any order of contact $\alpha$.  However, the choice $\kappa_0=\frac{1}{2}$ will work for many  compact semialgebraic sets $S$.

The $PL$  approximation $S_\ep$ of $S$ is obtained as before, using the two principles.   To prove  that the normal  cycle of $S_\ep$ converges in the sense of currents to the normal cycle of $S$ we rely on an approximation theorem  of J. Fu, \cite{Fu1}.    That  theorem     states that the convergence of the normal cycles is guaranteed once we prove two things.

\begin{itemize}
\item Uniform bounds for the perimeter   and total curvature of   $S_\ep$.

\item  For  almost any closed half-plane $H$ we have
\[
\lim_{\ep\searrow 0} \chi\bigl(\,   H\cap S_\ep\,\bigr)=\chi\bigl(\, H\cap S\,\bigr),
\]
where $\chi$ denotes the Euler characteristic.

\end{itemize}

Of the above two facts, the second is by far the most delicate, and its proof takes up the bulk of this paper.

Let us say a few words about the organization of the paper.  In Section \ref{s: 1} we introduce the terminology used throughout the paper.   {\bf Principle 1} is proved in Section \ref{s: 2}, while  the robustness principle ($\boldsymbol{R}$)   is proved in Section \ref{s: 3}.   

In Alorithm \ref{alg: process} of Section \ref{s: 4}    we give an explicit  and detailed  description  of the process that builds the approximation    $S_\ve$ starting from the pixelation $P_\ve(S)$. This section   contains    the proof of  the main result of the paper, Theorem \ref{th: main}, which states that the normal cycle of $S_\ve$ converges to the normal cycle of $S$ as $\ve\to 0$. 

The paper  concludes with  two appendices. In Appendix \ref{s: a} we  collect  a few  basic facts  of real algebraic geometry used throughout paper together with a few  other technical results.    In Appendix \ref{s: b} we give a more formal description of the approximation  algorithm  in a way that makes  it  easily implementable on a computer. 

\smallskip

\noindent {\bf Remark.} After this work was completed we became aware of a recent work \cite{CCLT} where the authors investigate a similar problem in arbitrary dimensions.  They  used a completely different approach to produce an algorithm for approximating the curvature measures of a  compact region $R$ in $\bR^n$ .   However the techniques used in \cite{CCLT}  apply only to regions  satisfying a so called  \emph{positive $\mu$-reach condition}. This condition prohibits the existence of cusp-like  singularities in $R$.   For example,  the techniques  in \cite{CCLT} are not apllicable to the region consisting of two tangent disks.  \qed

\medskip

\noindent {\bf Acknowledgment.}  We are very grateful   to the anonymous   referee for the  many  very useful and detailed comments, questions, suggestions and corrections which have  helped us improve the quality of the paper.

\section{Basic facts}
\label{s: 1}
\setcounter{equation}{0}

We begin by recalling  some basic  notions introduced in \cite{Row}.

\begin{definition} (a) Let $\ve>0$.  Then we define an \emph{$\ve$-pixel} to be the square in $\bR^2$ of the form
\[
S_{i,j}(\ve)=[(i-1)\ve, i\ve)]\times [(j-1)\ve, j\ve)]\subset \bR^2,\;\; i,j\in \bZ.
\]
Its center is
\[
c_{i,j}(\ve):=(i\ve,j\ve)-\left(\frac{\ve}{2},\frac{\ve}{2}\right).
\]  
\noindent (b) A union of finitely many $\ve$-pixels is called an $\ve$-\emph{pixelation}.  The variable $\ve$ is called the $\emph{resolution}$ of the pixelation.

\noindent (c) For any compact subset $S\subset \bR^2$ we define the $\ve$-\emph{pixelation} of $S$ to be the union of all the $\ve$-pixels that intersect $S$. We denote the $\ve$-pixelation of $S$ by $P_\ve(S)$. The pixelation of a function $f$ is defined to be the pixelation of  its graph $\Gamma(f)$. We will denote this pixelation by $P_\ve(f)$.\qed
\label{def: pix}
\end{definition}

\begin{definition}  Fix $\ve>0$ and a compact set $S\subset \bR^2$.

\begin{enumerate}

\item A point $a\in\bR$ will be called $\ep$-generic if $a\in\bR\setminus \ep\bZ$.   For such a point $a$ we denote by  $I_\ve(a)$ the interval  of the form $(\,n\ve, (n+1)\ve)$, $n\in\bZ$ that contains $a$. 

\item  For  $a<b$  we define the vertical strip
\[
\eS_{a,b} := (a,b) \times \bR
\]
For every $k\in \bZ$ we  denote by $\eS_{\ep,k}$ the vertical strip $ \eS_{k\ep, (k+1)\ep}$. For any $\ep$-generic point $a\in\bR$ we denote by $\eS_\ep(a)$ the strip   $S_{\ve,k}$, $k:=\lfloor a/\ep\rfloor$.

\item A \emph{column} of $P_\ve(S)$ is   the intersection  of $P_\ve(S)$ with a vertical strip $\eS_{\ep,k}$, $k\in\bZ$.  The connected components of  a column are called \emph{stacks}.

\item For every $\ep$-generic $a \in \bR$, we define the \emph{column} of a pixelation $P_\ep(S)$ over $a$ to be the set  
\[
C_\ve(S,a):=\eS_\ve(a)\cap P_\ve(S).
\]
In other words,   $C_\ep(S,a)$  is the union of the pixels in $P_\ep(S)$ which intersect the vertical  line $\{x=a\}$.   When $S$ is the graph of a function $f$, we will use the notation $C_\ep(f, a)$ to denote   the column over $a$ of the pixelation $P_\ve(f)$.

\end{enumerate}

 \qed
 \label{def: bu}
\end{definition}
We have the following result, \cite[Thm. 2.2]{Row}.

\begin{theorem} If $f:[a,b]\ra \bR$ is a continuous  function, then  for any $\ep>0$ the columns of the $\ep$-pixelation of  the graph of $f$ consist of single stacks.\qed
\label{thm: ivt}
\end{theorem}

In this  paper we will be concerned with pixelations of  generic  planar  semialgebraic sets, where the genericity has a very precise meaning.     To  describe it we need to introduce some terminology from stratified Morse theory, \cite{GM, Pig}.

For any subset $X\subset \bR^2$ we denote by $\cl(X)$ its \emph{closure} and by $\pato X$ its \emph{topological boundary},
\[
\pato(X):=\cl(X)\setminus X.
\]
We define  a \emph{good stratification} of  a compact semialgebraic set $S\subset \bR^2$ to be an increasing filtration 
\[
\eF:\;\;F^{(0)}\subset F^{(1)}\subset F^{(2)}=S
\]
satisfying the  following properties.

\begin{itemize}

\item Each of the sets $F^{(i)}$, $i=0,1,2$ is closed. 

\item  $\dim F^{(i)}\leq i$, $i=0,1,2$. In particular $F^{(0)}$ is a finite collection of points called the \emph{vertices} of the good stratification.

\item The connected components  of $F^{(1)}\setminus F^{(0)}$ are \emph{open real analytic arcs}, i.e.,  images of  injective real analytic maps $(0,1)\ra \bR^2$. We will refer to these components as the \emph{arcs} or the \emph{edges} of the stratification.

\item The connected components  of  $F^{(2)}\setminus F^{(1)}$ are open subsets of $\bR^2$. They are called the \emph{faces} of the stratification.

\item
\[
\pato \bigl( \,F^{(2)}\setminus F^{(1)}\,\bigr)\subset F^{(1)},\;\;\pato\bigl(F^{(1)}\setminus F^{(0)}\,\bigr)\subset F^{(0)}.
\]
\end{itemize}

\begin{definition}Suppose  that  $v$ is a vertex of a good stratification of a compact semialgebraic set $S\subset \bR^2$.  The \emph{tangent cone}  $C_\infty(v,S)$ to $S$ at $v$  consists of   finitely  many one-dimensional subspaces of $\bR^2$. More precisely, a line $L_\infty\subset \bR^2$  belongs to the tangent cone $C(v,S)$ iff  there exists an arc $A$ of the stratification of $S$  with the following properties.

\begin{itemize}

\item $v\in \cl(A)$.

\item There exists a sequence of points  $v_n\in A$ such that  as $n\ra \infty$  we have $v_n\ra v$  and the tangent spaces $T_{v_n}A$ converge to   $L_\infty$. 

\end{itemize}\qed
\label{def: tcone}
\end{definition}

Suppose that $S\subset \bR^2$ is a    compact semialgebraic set equipped with a good stratification $\eF$, and $f:\bR^2\ra \bR$.  A point $p\in S$ is said to be  a \emph{critical point} of the restriction $f|_S$  if either $p$ is a vertex, or $p$ is the critical point of the restriction  of $f$ to an arc or to a face.  The critical point  $p$ is  said to be \emph{nondegerate} if it satisfies one of the following conditions.

\begin{itemize}

\item[($\bsC_0$)] The point $p$ is a vertex and for any $L_\infty\in C_\infty(p,S)$,  the differential of $f$ at $p$ does not vanish along $L_\infty$.

\item[($\bsC_1$)]   The point $p$ belongs to an arc $A$ of the stratification and as such it is a nondegenerate point of $f|_A$.

\item[($\bsC_2$)] The point $p$ belongs   to a face $F$ of the stratification and as such it is a nondegenerate point  of $f|_F$.

\end{itemize}

A function $f: \bR^2\ra \bR$ is said to be a   a \emph{stratified Morse function} with respect to the semialgebaric set $S$ equipped with the good stratification $\eF$ if all  its critical    points are nondegenerate, and no two critical points  lie on the same level set of $f$.

A compact  semialgebraic set $S\subset \bR^2$ is  called \emph{generic} if  it admits a good stratification  $\eF$  such that projection onto the $x$-axis $(x,y)\mapsto x$ is a  stratified Morse function with respect to $(S,\eF)$.  Denote this projection by $\bh$.

Observe that if  $\eF$ is a good stratification of a compact semialgebraic  set $S\subset \bR^2$, then $p$ is a critical  point of $\bh$ relative to $(S,\eF)$ if either $p$ is a vertex of the stratification, or  $p$ is a  point on an arc of $\eF$ where the the tangent space  is vertical. At such a point the arc is locally on one side of that vertical tangent.

\begin{figure}[ht]
\centering{\includegraphics[height=0.8in,width=3in]{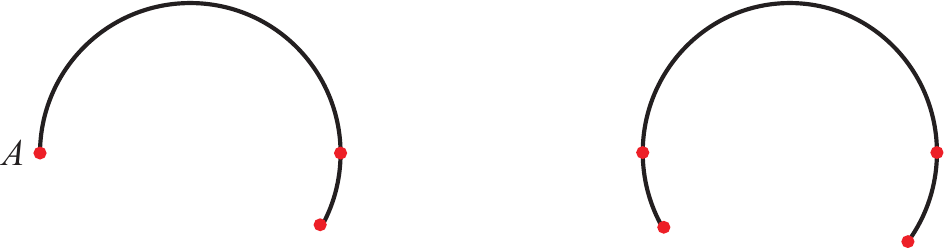}}
\caption{\sl  The curve on the right is generic, while the curve on the left is not.}
\label{fig: 1}
\end{figure}

In Figure \ref{fig: 1} we have depicted two  one-dimensional  planar semi algebraic curves (arcs of  circles). The marked points are    critical points of $\bh$. The point  $A$ on the left-hand side curve is a degenerate critical point because the condition ($\bsC_0$) is violated: the vertical line is tangent to the curve at that point.   

In the left-hand side of Figure \ref{fig: gen-nongen}   we have depicted    further  examples  of pathologies prohibited by the genericity   condition. (The pathologies  involve the points with vertical tangencies.) The  right-hand side   depicts  generic   sets  obtained by small perturbations from the nongeneric sets  in the left-hand side.

\begin{figure}[ht]
\centering{\includegraphics[height=1.4in,width=1.4in]{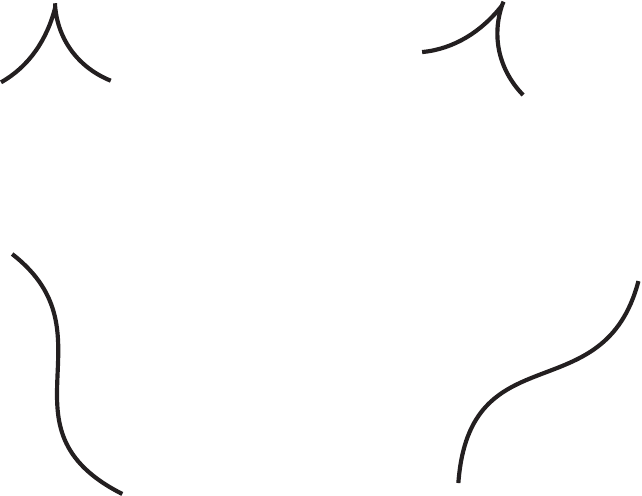}}
\caption{\sl The curves on the left-hand side  are nongeneric. They become generic after a small  perturbation.}
\label{fig: gen-nongen}
\end{figure}

\section{Approximations of elementary sets}
\label{s: 2}
\setcounter{equation}{0}

In this section we study the pixelations of simple two dimensional sets.  

\begin{definition} Fix $\ve>0$ and a compact set $S$. 

\begin{enumerate}
\item An $\ve$-\emph{profile} of $S$ is   a  set  $\Pi_\ve$   of points in the plane  with the following properties.

\begin{enumerate} 

\item Each point in  $\Pi_\ve$ is the center of  an $\ve$-pixel that intersects $S$. 

\item   Every  column of $P_\ve(S)$ contains precisely one point of $\Pi_\ve$.
\end{enumerate}

\item The  \emph{top/bottom $\ve$-profile} is   the profile  consisting of the centers  of  the highest/lowest pixels in  each column of $P_\ve(S)$.

\item An $\ve$-sample of $S$ is a subset of an $\ve$-profile. An \emph{upper/lower}  $\ep$-sample of $S$ is an  $\ep$-sample of the upper/lower $\ep$-profile of $S$.

\item Two $\ve$-samples are  called \emph{compatible} if they have the same projections on the $x$-axis.

\end{enumerate}\qed
\label{def: profile}
\end{definition}

\begin{definition} Suppose $p_1,\dotsc, p_N$ is a finite sequence of points in $\bR^2$.  (The points need not be distinct). We denote by
\[
\langle p_1,p_2,\dotsc, p_n\rangle
\]
the $PL$ curve defined  as the union of the straight line segments $[p_1,p_2],\dotsc, [p_{n-1},p_n]$.\qed
\end{definition}

Observe that   each $\ve$-profile $\Pi_\ve$ of a set is equipped with a linear order $\preceq$. More precisely, if $p_1,p_2$ are points in $\Pi_\ve$, then 
\[
p_1\preceq p_2\Longleftrightarrow x(p_1)\leq x(p_2),
\]
where $x:\bR^2\ra \bR$   denotes the projection $(x,y)\mapsto x$.  In particular, this shows that  any $\ve$-sample of $S$ carries a natural total order.

\begin{definition}If  $\Xi$  is an $\ve$-sample of $S$,  then the $PL$-interpolation  determined by the sample $\Xi$ is the continuous, piecewise   linear function  $L= L_\Xi$ obtained as follows.

\begin{itemize}

\item Arrange the points in $\Xi$ in increasing order, with respect to the above total order,
\[
V=\{ \xi_0\prec \xi_1\prec\xi_2\prec\cdots \prec \xi_n\},\;\;n+1=|\Xi|.
\]
\item The graph of $L_\Xi$ is the $PL$-curve $\langle \xi_0,\xi_1, \dotsc, \xi_n\rangle$. 

\end{itemize}\qed
\label{def: interpolate}
\end{definition} 

In applications, the sample sets $\Xi$ will be chosen to satisfy   certain regularity.    

\begin{definition} \begin{enumerate}
\item A \emph{spread function} is a nonincreasing function $\si: (0,\infty)\to \bZ_{>0}$  with the following properties.
\begin{subequations}
\begin{equation}
\lim_{\ep \searrow 0} \sigma(\ep) = \infty,
\label{eq: spread1}
\end{equation}
\begin{equation}
\lim_{\ep \searrow 0} \ep \sigma(\ep) = 0.
\label{eq: spread2}
\end{equation}
\end{subequations}

\item If $\sigma $ is a positive  integer and $\Pi_\ep$ is an $\ep$-profile, then an \emph{$\ve$-sample with spread $\sigma$} is a subset  
\[
\Xi=\{\xi_0\prec\cdots \prec\xi_n\}\subset \Pi_\ve(S)
\]
such that  the following hold.

\begin{itemize}

\item The points $\xi_0$ and $\xi_n$ are the left and rightmost points in the profile, i.e., for each $p \in \Pi_\ep$, $x(\xi_0) \le x(p) \le x(\xi_n)$.

\item 
\[
\frac{1}{2}\ve \sigma(\ve  \leq |x(\xi_k)-x(\xi_{k-1})|\leq \ve \sigma(\ve),\;\;\forall k=1,\dotsc, n.
\]
\end{itemize}
\end{enumerate}\qed
\end{definition}

\begin{definition} A subset  $S\subset \bR^2$ is  said to be  \emph{elementary} over the interval $[a,b]$  if it can be   defined as
\[
S=S(\beta,\tau):= \bigl\{\, (x,y):\;\; x \in [a,b], \beta(x) \le y \le \tau(x)\,\bigr\} ,
\]
where $\beta,\tau:[a,b]\ra \bR$ are  continuous semialgebraic functions  such that  $\beta(x)\leq \tau(x)$, $\forall x\in [a,b]$. 

The function $\beta$ is called the \emph{bottom} of $S$ while $\tau$ is called the \emph{top} of $S$.  If 
\begin{equation}
\beta(x) < \tau(x),\;\;\forall x\in (a,b),
\label{eq: bt1}
\end{equation}
then the elementary set is  said to be \emph{nondegenerate}. If 
\begin{equation}
\beta(x)=\tau(x),\;\;\forall x\in [a,b],
\label{eq: bt2}
\end{equation}
    then the set $S$ is called \emph{degenerate}.  The elementary set is called \emph{mixed}  if   both sets 
    \[
    \bigl\{ x\in (a,b);\;\;\tau(x)-\beta(x)>0\,\bigr\}\;\;\mbox{and}\;\; \bigl\{ x\in (a,b);\;\;\tau(x)-\beta(x)=0\,\bigr\}
    \]
    are nonempty.\qed
\label{def: s-type}
\end{definition}

Observe that an elementary set $S(\beta,\tau)$ over the compact interval $[a,b]$ admits \emph{good partitions}, i.e., partitions
\[
a=c_0<c_1<c_1<\cdots <c_n=b,\;\;n\geq 2
\]
such that each of the elementary sets $[c_{i-1},c_i]\times \bR\cap S(\beta,\tau)$ is either degenerate or  nondegenerate.   The good partition  with minimal $n$ is called the \emph{minimal good partition}; see Figure \ref{fig: good-part}.

\begin{figure}[ht]
\centering{\includegraphics[height=1.3in,width=3.5in]{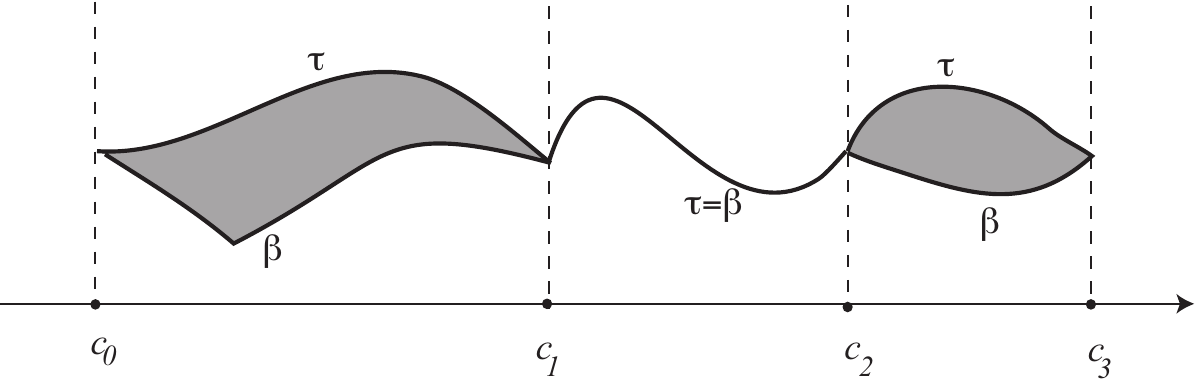}}
\caption{\sl The minimal good partition of a mixed elementary set.}
\label{fig: good-part}
\end{figure}

In the remainder of this section $S$ will indicate an elementary set.  We first note that like the pixelation of a function, each column of $P_\ep(S)$ consists of a single stack, i.e., a single connected component.

\begin{proposition}
If $S=S(\beta,\tau)$ is an elementary set over a compact interval $[a,b]$,  then for every  $x \in [a,b]\setminus \ep\bZ$, the column $C_\ep(S,x)$ consists of exactly one stack.
\label{pro: ivt2}
\end{proposition}

\begin{proof}  Fix an $\ep$-generic $x\in [a,b]$. By Theorem \ref{thm: ivt} the columns  $C_\ve(\beta,x)$  and    $C_\ep(\tau,x)$    consist of single stacks. If these two columns intersect, then  the conclusion is obvious. If they do not intersect, then  any  pixel in the strip $S_\ve(x)$ situated below the stack $C_\ve(\tau,x)$ and above    the stack $C_\ve(\beta,x)$ is a pixel of $P_\ve(S)$.  This  proves that the  column $C_\ve(S,x)$ consists of a single stack. \end{proof}

\begin{definition}Fix $\ep>0$ and an elementary  set $S=S(\beta,\tau)$.   Suppose that $\Xi^\pm_\ep$ are compatible upper/lower samples of $S$
\[
\Xi^\pm_\ve=\{\xi_0^\pm\prec\xi_1^\pm\prec\cdots \prec\xi_n^\pm\}.
\]
The  $PL$-\emph{approximation}  of $S$  determined by these two  samples is the compact $PL$-set   bounded by the  closed $PL$-curve
\[
\langle \xi_0^-, \xi_1^-,\dotsc, \xi_n^-,\xi_n^+,\xi_{n-1}^+,\dotsc,\xi_0^+, \xi_0^-\rangle.
\]
\qed

\end{definition}

The total curvature of a $C^1$-immersion $\gamma:[a,b]\ra \bR^2$ which is $C^2$  on $(a,b)$  is defined as follows.    We set 
\[
\bv(t):=\frac{1}{|\dot{\gamma}(t)|}\dot{\gamma(t)},\;\;\kappa(t):=|\dot{\bv}(t)|.
\]
The scalar $\kappa(t)$ is called the \emph{curvature} of $\gamma$ at the point $\gamma(t)$.  We define the \emph{total curvature}  of $\gamma$ to be
\[
K(\gamma):=\int_a^b \kappa(t)\,dt.
\]
Suppose now that  $\gamma:[a,b]\ra \bR^2$ is a continuous  and  piecewise $C^2$-immersion, i.e., there exists a finite subset $\{t_0,\dotsc, t_\nu\}\subset [a,b]$,  such that
\[
a=t_0<t_1<\cdots <t_{\nu-1}<t_\nu=b,
\]
 the restriction $\gamma_i:=\gamma|_{[t_{i-1},t_i]}$ is  a $C^1$ immersion,  and the restriction of  $\gamma$ to the open interval to $(t_{i-1},t_i)$ is $C^2$, for any $i=1,\dotsc,\nu$.  The curvature of  $\gamma$ at a jump point $\gamma(t_i)$ is the quantity
\[
\kappa(t_i):=\lim_{\ve\searrow 0} \dist_{S^1}\bigl(\, \bv(t_i+\ve)\,, \,\bv(t_i-\ve)\,\bigr),
\]
where $\dist_{S^1}(p,q)\in[0,\pi]$ denotes the geodesic distance between two points $p,q$ on the unit circle. We  define the total curvature of  $\gamma$ to be
\[
\begin{split}
K(\gamma)&:=\sum_{i=1}^\nu K(\gamma_i)+\sum_{i=1}^{\nu-1}\kappa(t_i)\\
&+\begin{cases}0, & \gamma(b)\neq\gamma(a),\\
& \\
\dist_{S_1}\bigl(\,\bv(b^-),\bv(a^+\,\bigr), &\gamma(b)=\gamma(a).
\end{cases}
\end{split}
\]
For more details we refer to \cite{Mil} and \cite[\S 2.2]{Mor}.

We define a \emph{semialgebraic arc}  to be the image  of a continuous, injective semialgebraic map 
\[
\vfi:[a,b]\ra \bR^2. 
\]
Suppose that $\vfi:[a,b]\ra \bR^2$ is a continuous, injective, semialgebraic map  whose image  is the semialgebraic arc $\bsC$. Set $A:=\vfi(a)$ and $B:=\vfi(b)$ so that $\bsC$ connects $A$ to $B$.  The  semialgebraic map $\vfi$ is piecewise $C^2$ and has a total curvature which a priori could be infinite. 

\begin{lemma} The total  curvature of a semialgebraic arc $C\subset \bR^2$ is finite.
\end{lemma}

\begin{proof}  The arc $C$ has finitely many singularities  and their complement is  a finite union of oriented,  bounded, semialgebraic,  $C^1$-arcs.  If $C_i$ is such an arc,  then its total curvature is   the length of the   oriented Gauss  path $\Gamma_i:C_i\to S^1$, where $\Gamma_i(p)$ is the  unique unit vector in $T_p C_i$    determined by the orientation of $C_i$.  Since the Gauss map is semialgebraic and $C_i$ has finite length  we deduce that the length of $\Gamma_i$ is finite.    The contributions to the total curvatures of the finitely many singular points  of $C$  are all finite.
\end{proof}

Suppose that $\bsC$ is a semialgebraic arc  defined  by the  the continuous, semialgebraic injection $\vfi:[a,b]\to\bR^2$. An \emph{ordered sampling} of $\bsC$ is an ordered  collection of points 
\[
\eP:=\bigl\{ P_1,\dotsc, P_n\,\bigr\}\subset  \bsC,
\]
such that  the collection
\[
\vfi^{-1}(\eP):=\bigl\{ t_1=\vfi^{-1}(P_1),\dotsc, t_n:=\vfi^{-1}(P_n)\bigr\}\subset [a,b]
\]
satisfies
\[
t_1<t_2<\cdots<t_n.
\]
The \emph{mesh} of the ordered sampling $\eP$ is the positive  number
\[
\|\eP\|:=\max \bigl\{ \, \dist(A,P_1),\dist(P_1,P_2),\dotsc, \dist(P_{n-1},P_n),\dist(P_n,B)\,\bigr\}.
\]
We denote by $\bsC(\eP)$ the $PL$-curve $\lan P_1,\dotsc, P_n\ran$.   

A result of J. Milnor,  \cite[Thm.2.2]{Mil},   shows that the total curvature   of a $C^2$-curve  can be  approximated by the  total curvature of inscribed   polygons. The next result, whose proof is delegated to Appendix \ref{s: a}, shows that  if the curve is semialgebraic, then  the $C^2$ requirement  is not necessary. 

\begin{proposition}  Suppose   that $\bsC$ is a semialgebraic  arc  and for every $\ep>0$ we are given an ordered sampling $\eP_\ep$ of $\bsC$.    Denote by $L$ (resp. $K$) the  length (resp. total curvature) of $\bsC$ and by $L_\ve$ (resp. $K_\ve$) the length (resp. total curvature) of $\bsC(\eP_\ve)$.  If
\[
\lim_{\ve\searrow 0}\|\eP_\ve\|=0,
\]
then
\[
\lim_{\ve\searrow 0} L_\ve =L\;\;\mbox{and}\;\; \lim_{\ve\searrow 0} K_\ve =K.
\]\qed
\label{prop: pl-approx}
\end{proposition}

\begin{theorem}  Suppose that $h: [a,b]\ra \bR$ is a continuous semialgebraic function and $\ve\mapsto \si(\ep)$ is a spread function satisfying the additional condition
\begin{equation}
\lim_{\ve\searrow 0} \ve\si(\ve)^2=\infty.
\label{eq: spread3}
\end{equation}
 For every $\ve>0$ we choose  an $\ep$-sample $\Xi_\ep$ with spread $\si(\ep)$ of the graph $\Gamma$ of $h$.   Denote by $\bsC_\ep$ the graph of the  $PL$-function $L_{\Xi_\ep}$ described in Definition \ref{def: interpolate}. Then, as $\ve\searrow 0$ we have
\[
{\rm length}\,(\bsC_\ep)\ra {\rm length}\,(\Gamma)\;\;\mbox{and}\;\; K(\bsC_\ep)\ra K(\Gamma).
\]
\label{th: pl-approx1}
\end{theorem}

\begin{proof} We use a simple strategy.  More precisely   for every $\ep>0$ we construct an ordered sampling $\eP_\ep$ of $\Gamma$ such that    
\begin{equation}
\lim_{\ve\searrow 0} \|\eP_\ep\|=0,
\label{eq: mesh}
\end{equation}
and if $\Gamma_\ep$ denotes the $PL$-curve $\Gamma(\eP_\ep)$ determined by the ordered sampling $\eP_\ep$,  then as $\ep\searrow 0$ we have
\begin{subequations}
\begin{equation}
{\rm length}\,(\bsC_\ep)= {\rm length}\,(\Gamma_\ep)+ O\left(\frac{1}{\si(\ep)}\right),
\label{eq: app-length}
\end{equation}
\begin{equation}
K(\bsC_\ep)=K(\Gamma_\ep) +O\left(\frac{1}{\ep\si(\ep)^2}\right).
\label{eq: app-curv}
\end{equation}
\end{subequations}
The desired conclusions  will then follow from Proposition \ref{prop: pl-approx}.

Suppose  that $\Xi_\ep$ consists of the points $Q_0^\ep,Q_1^\ep,\dotsc, Q_{n(\ep)}^\ep$ arranged  in  the increasing order defined by their  $x$-coordinates.    Observe that  since $\Xi_\ep$ has spread  $\si(\ep)$ then
\begin{equation}
n(\ep)<\frac{2(b-a)}{\ep\si(\ep)}.
\label{eq: subdivisions}
\end{equation}
Each of the points of $\Xi_\ep$  is the center of a pixel that touches $\Gamma$. Thus, for any  $k=0,1,\dotsc, n(\ep)$ there exists a point $P_k^\ep\in \Gamma$ that lies in the same pixel as $Q_k^\ep$. We obtain in this fashion an ordered sampling
\[
\eP_\ep=\bigl\{ P_0^\ep,P_1^\ep,\dotsc, P_{n(\ep)}^\ep\,\bigr\}
\]
of  $\Gamma$.    The  function $h$ is continuous and  semialgebraic and thus it is H\"{o}lder continuous with some H\"{o}lder exponent $\alpha\in (0,1]$.   This proves that
\[
\|\eP_\ep\|= O\bigl( \;\bigl(\,\ep\si(\ep)\;\bigr)^\alpha\,\bigr).
\]
The condition (\ref{eq: mesh}) now follows from the property (\ref{eq: spread2}) of a spread function.

From the choice of the points $P_k^\ep$ we deduce  that for any $k=1,\dotsc, n(\ep)$ we have
\[
-\ep\sqrt{2}<\dist(P_{k-1}^\ep,P_k^\ep)-\dist(Q_{k-1}^\ep,Q_k^\ep)  <\ep\sqrt{2}
\]
so that  by summing over $k$ we deduce 
\[
-\frac{2\sqrt{2}(b-a)}{\si(\ep)}\stackrel{(\ref{eq: subdivisions})}{\leq} -n(\ep)\ep\sqrt{2} <{\rm length}\,(\Gamma_\ep)- {\rm length}\,(\bsC_\ep)\leq n(\ep)\ep\sqrt{2}\stackrel{(\ref{eq: subdivisions})}{\leq} \frac{2\sqrt{2}(b-a)}{\si(\ep)}.
\]
The equality (\ref{eq: app-length}) now follows from the property (\ref{eq: spread1}) of a spread function.

Now we turn to the total curvature.  For any point $ P\in\bR^2$ we denote by $x(P)$, $y(P)$ its coordinates. For $k=1,\dotsc, n(\ep)$ we denote by $m_k^\ep$ the slope  of the segment $[P_{k-1}^\ep,P_k^\ep]$ and   by $\bar{m}_k^\ep$ the slope of the segment $[Q_{k-1}^\ep,Q_k^\ep]$,
\[
m_k^\ep=\frac{y(P_k^\ep)-y(P_{k-1}^\ep)}{x(P_k^\ep)-x(P_{k-1}^\ep)},\;\;\bar{m}_k^\ep=\frac{y(Q_k^\ep)-y(Q_{k-1}^\ep)}{x(Q_k^\ep)-x(Q_{k-1}^\ep)}.
\]
  Note that  for any $\ep>0$ and any $k=1,\dotsc, n(\ep)$ we have from the definition of a spread that
\[
\frac{1}{2}\ep\si(\ep)\leq  x\bigl(\, Q_{k}^\ep\,\bigr)-x\bigl(\,Q_{k-1}^\ep\,\bigr)\,\le \ep\si(\ep).
\]
Furthermore we have shown that
\[
\dist(Q_k^\ep, P_k^\ep) \le \ep \sqrt{2}.
\]
These two inequalities imply that
\[
|m_k^\ep- \bar{m}_k^\ep| = O\left(\frac{1}{\sigma(\ep)}\right)
\]
There exist  $\theta_k^\ep,\bar{\theta}^\ep_k\in (-\frac{\pi}{2},\frac{\pi}{2})$ such that
\[
m_k^\ep=\tan \theta^\ep_k,\;\;\bar{m}^\ep_k=\tan \bar{\theta}^\ep_k.
\]
The  formula for the tangent of a difference of angles  implies that
\[
\theta_k^\ep - \bar{\theta}_k^\ep = \arctan \left( \frac{m_k^\ep - \bar{m}_k^\ep}{1 + m_k^\ep \bar{m}_k^\ep} \right).
\]
Using the above equality and the fact that $|m_k^\ep- \bar{m}_k^\ep| = O\left(\frac{1}{\sigma(\ep)}\right)$, we see that there exists a positive constant $C$ independent of $\ep$ such that 
\[
|\theta_k^\ep - \bar{\theta}_k^\ep| \le \frac{C}{\sigma(\ep)}
\]
and therefore
\[
\bigl|\, K(\Gamma_\ep)-K(\bsC_\ep)\,\bigr|\le \frac{Cn(\ep)}{\si(\ep)}\stackrel{(\ref{eq: subdivisions})}{\leq} \frac{C(b-a)}{\ep\si(\ep)^2}.
\]
The  equality (\ref{eq: app-curv}) now follows from (\ref{eq: spread3}).
\end{proof}

\begin{corollary}
Let $S(\beta,\tau)$ be an  elementary set. Fix a spread function $\si(\ep)$  satisfying  the condition (\ref{eq: spread3}). For each $\ep$ we choose  compatible $\ep$-upper/lower profiles  $\Xi^\pm_\ep$ of $S$ with spread $\si(\ep)$. We denote by ${S}_\ep$ the $PL$ approximation of ${S}$ defined by these samples. Then
\begin{subequations}
\begin{equation}
\lim_{\ep \searrow 0}{\rm length}\,(\pa S_\ep)= {\rm length}\,(\pa S),
\label{eq: length}
\end{equation}
\begin{equation}
\lim_{\ep \searrow 0} K(\partial S_\ep) = K(\partial S).
\label{eq: curv}
\end{equation}
\end{subequations}
\label{cor: piece curv}
\end{corollary}

\begin{proof}  The   semialgebraic functions $\beta$ and $\tau$ are differentiable  outside a finite  subset of $(a,b)$ and the limits 
\[
\beta'(a):=\lim_{x\searrow a} \beta'(x), \;\; \tau'(a):=\lim_{x\searrow a} \tau'(x),
\]
\[
\beta'(b):=\lim_{x\nearrow b} \beta'(x), \;\; \tau'(b):=\lim_{x\nearrow b} \tau'(x)
\]
exist in $[-\infty,\infty]$.  Let
\[
\Xi^\pm_\ep=\bigl\{ \xi_0^\pm \prec\xi_1^\pm \prec\cdots \prec\xi_{n(\ep)}^\pm\,\bigr\},\;\;\xi_k^\pm=: (x_k^\pm, y_k^\pm).
\]
The compatibility condition implies that
\[
x_k^-=x_k^+=:x_k,\;\;y_k^-\leq y_k^+,\;\;\forall k=0,1,\dotsc, n.
\]
Let $\beta_\ep$ be the PL function whose graph is $L_{\Xi_\ep^-}$  and $\tau_\ep$  be the PL function whose graph is $L_{\Xi_\ep^+}$.   Let $m_i^\beta(\ep)$ indicate the slope of the $i$-th line segment of the graph of  $\beta_\ep$ and similarly let $m_i^\tau(\ep)$ indicate the slope of the $i$-th line segment of the graph of $\tau_\ep$. We have
\[
{\rm length}(\pa S_\ep) ={\rm length}\,(\Gamma_{\beta_\ep})+{\rm length}\,(\Gamma_{\tau_\ep})+\dist(\xi_0^-,\xi_0^+)+\dist (\xi_{n(\ep)}^-,\xi_{n(\ep)}^+\,).
\]
Theorem \ref{th: pl-approx1} implies that as $\ep\searrow 0$ we have
\[
{\rm length}\,(\Gamma_{\beta_\ep})\ra {\rm length}\,(\Gamma_{\beta}),\;\; {\rm length}\,(\Gamma_{\tau_\ep})\ra {\rm length}\,(\Gamma_{\tau}).
\]
Moreover, as $\ve\searrow 0$ we have
\[
\dist(\xi_0^-,\xi_0^+)\ra \dist(\beta(a),\tau(a)\,),\;\;\dist (\xi_{n(\ep)}^-,\xi_{n(\ep)}^+\,)\ra \dist(\,\beta(b),\tau(b)\,).
\]
This proves (\ref{eq: length}).

Similarly
\[
\begin{split}
K(\partial S_\ep) & = |\pi - \arctan(m_1^\beta(\ep)\,)| \\
&+ \sum_{i=2}^n |\arctan(m_i^\beta(\ep)\,) - \arctan(m_{i-1}^\beta(\ep)\,)| \\
&+ |\arctan(m_n^\beta(\ep)) - \pi| + |\pi - \arctan(m_1^\tau)| \\
&+ \sum_{i=2}^n |\arctan(m_i^\tau(\ep)\,) - \arctan(m_{i-1}^\tau(\ep)\,)| \\
&+ |\arctan(m_n^\tau(\ep)\,) - \pi|
\end{split}
\]
which can be rewritten as
\begin{equation}
\begin{split}
K(\partial S_\ep) & = |\pi - \arctan(m_1^\beta(\ep)\,)| + |\arctan(m_n^\beta(\ep)\,) - \pi| \\
&+ |\pi - \arctan(m_1^\tau(\ep)\,)| + |\arctan(m_n^\tau(\ep)\,) - \pi| \\
&+ K(\beta_\ep) + K(\tau_\ep)
\end{split}
\label{eq: epcurv}
\end{equation}
Theorem  \ref{th: pl-approx1} implies
\begin{equation}
\lim_{\ep \searrow 0} K(\beta_\ep) = K(\beta)\;\;\mbox{and}\;\; \lim_{\ep \searrow 0} K(\tau_\ep) = K(\tau).
\label{eq: piece conv}
\end{equation}
Now note that each line segment is defined by connecting two points in the pixelation of $\beta$ or $\tau$ over an interval of at most $\ep \sigma(\ep)$. As $\ep \to 0$ we have
\begin{equation}
\begin{split}
\lim_{\ep \searrow 0} m_1^\beta(\ep) = \beta'(a), \;\;\lim_{\ep \searrow 0} m_n^\beta(\ep) = \beta'(b),\\
\lim_{\ep \searrow 0} m_1^\tau(\ep) = \tau'(a), \;\;\lim_{\ep \searrow 0} m_n^\tau(\ep) = \beta'(b).
\end{split} 
\label{eq: slope convergence}
\end{equation}
Combining (\ref{eq: epcurv}), (\ref{eq: piece conv}), (\ref{eq: slope convergence}) we find that
\begin{equation}
\begin{split}
\lim_{\ep \searrow 0} K(\partial S_\ep) &= |\pi - \arctan(\beta'(a))| + |\arctan(\beta'(b)) - \pi| \\
& + |\pi - \arctan(\tau'(a))| + |\arctan(\tau'(b)) - \pi| \\
&+ K(\beta) + K(\tau).
\end{split}
\label{eq: final curv}
\end{equation}
Note that $|\pi - \arctan(\beta'(a))|$ is the value of the angle between the vertical line $x = a$ and the tangent line to the graph of  $\beta$ at $(a,\beta(a))$.  Similarly each other difference on the right hand side of (\ref{eq: final curv}) corresponds to an angle at one of the corners of $\partial S$.  Therefore the right hand side of the (\ref{eq: final curv}) is equal to the $K(\partial S)$, so the corollary holds.
\end{proof}

\section{Separation results}
\label{s: 3}
\setcounter{equation}{0}

In the previous section we have dealt only with  the elementary regions and we have investigated mainly  \emph{geometric} properties of these regions and their pixelation.  In this section we turn our attention to the relationship between the topologies of a semialgebraic set  and those of its pixelations.   

Surprisingly, this is a nontrivial matter. As shown in the introduction the homotopy type of a   planar set  may  be quite different from those of its pixelations and this  can happen even for a simple $PL$  set. The next result   provides a first ray of hope.  For any compact set $X\subset \bR^2$ we denote by $\eC(X)$ the  set of   connected components of $X$.

\begin{proposition}
Let $S \subset \bR^2$ be a compact  semialgebraic  set.  Then for sufficiently small $\ep$, the number of connected components of $P_\ep(S)$ agrees with the number of connected components of $S$.
\label{prop: separate}
\end{proposition}

\begin{proof}  We have a natural map $\eC(S)\ra \eC(\,P_\ep(S)\,)$ that associates to each connected component $C$  of $S$ the  unique connected component  of $P_\ep(S)$ containing $C$. For $\ep$ sufficiently small this map is injective. Since $P_\ep(S)$ contains only pixels that intersect $S$, we deduce that each connected component of  $P_\ep(S)$ contains at least one connected component of $S$. Thus,  $P_\ep(S)$ has  at most as many components as $S$.
\end{proof}

The  above  result  guarantees that   the zeroth Betti number of a compact semialgebraic set coincides with those of its sufficiently fine pixelations.  Proposition \ref{prop: separate} also suggests that, for small $\ep$, the only way that the homotopy  type of $P_\ep(S)$ can disagree with that of $S$ is if $P_\ep(S)$ has  \emph{holes}, i.e.,  cycles in $P_\ep(S)$ that are not contained in the image of the inclusion induced morphism
\[
H_1(S,\bZ)\ra H_1\bigl(\, P_\ep(S),\bZ\,\bigr).
\]
Thus, the recovery of $S$ from $P_\ep(S)$ will depend on distinguishing the cycles of $P_\ep(S)$ that correspond to real cycles from $S$ from those that are merely artifacts of the pixelation. To discard  these holes, we  adopt a strategy inspired from Morse theory.

\begin{definition}
Let $S\subset \bR^2$ be a compact set and $\ep>0$.

\begin{enumerate}

\item For every $\ep$-generic $x$ we set  
\[
\bn_\ep(x) =\bn_{S,\ep}:= \# \text{ of connected components of }C_\ep(S,x).
\]
(When the  set $S$  is  understood from context we use the simpler notation $\bn_\ep$ instead of $\bn_{S,\ep}$.) We will refer to $\bn_\ep(x)$ as the \emph{stack counter function} of $S$.

\item For any $x_0\in \bR$  we define 
\[
\bn(x_0) =\bn_S(x):= \# \text{ of components of  the intersection of $S$ with the vertical line $x=x_0$}.
\]
We will refer to $\bn_S$ the \emph{component counter} of  $S$.

\item A \emph{jumping point}  of $\bn_S$ is a  real number $x_0$ such that  
\[
\bn_S(x_0)\neq \bn_S(x_0^-):=\lim_{x\nearrow x_0} \bn_S(x) \;\;\mbox{or}\;\;\bn_S(x)\neq \bn_S(x_0^+):=\lim_{x\searrow x_0}\bn_S(x).
\]
 We denote by $\eJ_S$ the set of jumping points of $\bn_S$. We will refer to $\eJ_S$ as the \emph{jumping set} of $S$.

\item A \emph{jumping point} of $\bn_\ep$ is a real number $x_0\in\ep\bZ+\frac{\ep}{2}$ such that  
\[
\bn_\ep(x_0-\ep)\neq \bn_\ep(x_0).
\]
 We denote by $\eJ_{S,\ep}$ the set of jumping points of $\bn_{S,\ep}$.  We will refer to it as the $\ep$-\emph{jumping set} of $S$.
\end{enumerate}\qed
\end{definition}

Let us point out that if $S$ is semialgebraic, then its jumping set $\eJ_S$ is finite and it is contained in the set of critical values of the restriction to $S$ of the function $\bh(x,y)=x$. The function $\bn_\ep$ tells us how many stacks are in a column.   The jumps  of $\bn_\ep$ are a first indicator of the presence of    cycles in $P_\ep(S)$.  To decide  whether they are    holes, as opposed to cycles coming from $S$  we will rely on the next key technical result.

\begin{theorem}[Separation Theorem]
Let $f,g:[a,b]\ra \bR$ be  two semialgebraic continuous  functions such that $f(x) < g(x)$, $\forall x\in[a,b]$.   Denote by $G$ the union of the graphs of $f$ and $g$.  Fix $L>0$, $\alpha\in (0,1]$  and $x_0\in [a,b]$ such that either
\begin{equation}
|g(x)-g(y)| \leq L|x-y|^\alpha,\;\; \text{ or }\;\; |f(x)-f(y)| \leq L|x-y|^\alpha,\;\; \forall x,y\in [a,b]
\label{eq: holder}
\end{equation}
and 
\begin{equation}
 g(x_0) - f(x_0 )\leq g(x)-f(x),\;\;\forall x\in[a,b],
\label{eq: gap0}
\end{equation}
Then  for any  $\ep>0$ such that
\begin{equation}
3\ep +L\ep^\alpha <g(x_0) - f(x_0 )
\label{eq: gap1}
\end{equation}
and any  $\ep$-generic $x \in [a,b] \setminus \ep \bZ$  the column $C_\ep(G,x_0)$ has two components.  In other words,  if
\[
\min_{x\in[a,b]}\bigl(\, g(x)-f(x)\,\bigr)\geq 3\ep +L\ep^\alpha,  
\]
then for any $\ep$-generic $x\in [a ,b]$  we have
\[
\bn_{G,\ep}(x)=\bn_G(x).
\]
\label{th: sep}
\end{theorem}

\begin{proof} We deal with the case that $|g(x)-g(y)| \leq L|x-y|^\alpha$. For any $\ep$-generic $x$ we denote by $T_\ep(f,x)$ (resp. $B_\ep(f,x)$) the  altitude of the center of the top (resp. bottom) pixel of the column $C_\ep(f,x)$. $B_\ve(g,x)$ and $T_\ve(g,x)$ are defined similarly. 

For a $\ve$-generic $x$ we have
\[
C_\ep(G, x) = C_\ep(f, x) \cup C_\ep(g, x),
\]
and furthermore,  Theorem \ref{thm: ivt} implies that each of these columns is connected.  Therefore $C_\ep(G, x)$ will have two components exactly when the columns  $C_\ep(f,x)$ and $C_\ep(g,x)$ do not intersect.   Since $f \le g$ and $x$ is $\ep$-generic, this will occur when 
\[
T_\ep(f,x) < B_\ep(g,x),
\]
or equivalently,
\[
B_\ep(g, x) - T_\ep(f,x) >\ep.
\]
Now fix $x\in[a,b]$ and let $i \in \bZ$ such that $i \ep < x < (i+1)\ep$.   Choose $x_f,x_g\in [i\ep, (i+1)\ep
]$ such that
\[
f(x_f)=\max_{x\in [i\ep,(i+1)\ep]} f(x),\;\; g(x_g)= \min_{x\in [i\ep,(i+1)\ep]} g(x).
\]
Therefore we have
\[
B_\ep(g,x)\geq g(x_g)-\ep,\;\;  T_\ep(f,x)\leq f(\,x_f\,)+\ep,
\]
so that
\[
B_\ep(g,x)-T_\ep(f,x)\geq  g(x_g)-f(x_f) -2\ep.
\]
Thus
\[
B_\ep(g,x)-T_\ep(f,x)\geq g(x_g)-g(x_f)+g(\,x_f\,)-f(\,x_f\, ) -2\ep
\]
\[
\stackrel{(\ref{eq: gap0})}{\geq}  g(x_g)-g(x_f)+ g(x_0)-f(x_0)-2\ep
\]
\[
\stackrel{(\ref{eq: holder})}{\geq} g(x_0)-f(x_0)-L\ep^\alpha -2\ep\stackrel{(\ref{eq: gap1})}{>}\ep.
\]
which completes the proof for the case that $|g(x)-g(y)| \leq L|x-y|^\alpha$.  

The case when  $|f(x)-f(y)| \leq L|x-y|^\alpha$  can be obtained from the above case  by working with a new pair of functions $g_1=-f$ and $f_1=-g$.
\end{proof}

This important Separation Theorem can be used to prove the following two results, which will tell us exactly when $\bn$ and $\bn_\ep$ correspond, using only information from the pixelation.  With these theorems we will be able to distinguish real cycles from the original set from fake cycles created by the pixelation.

\begin{theorem}
Let $S\subset \bR^2$ be a compact semialgebraic set with jumping set $\eJ_S$. Then there exist $\kappa_0=\kappa_0(S)\in (0,1]$, $\nu_0=\nu_0(S)>0$,  $\ep_0=\ep_0(S)>0$, \emph{depending only on $S$}, such that,  if  $0<\ep <\ep_0$ and $x$ is $\ep$-generic  and satisfies 
\[
\dist(x,\eJ_S)\geq \nu_0\ep^{\kappa_0},
\]
 then $\bn_{S,\ep}(x) = \bn_S(x)$.
\label{thm: noisebound}
\end{theorem}

\begin{proof} Let  $x_0<x_1<\cdots <x_\ell$ be the jumping points of  $\bn=\bn_S$.   We set
\[
\Delta x_i:=x_i-x_{i-1},\;\;\forall i=1,\dotsc, \ell,\;\;\Delta:=\min_{1\leq i\leq \ell} \Delta x_i.
\]
 Note that $\bn(x)$ is constant on each of the intervals $(x_{i-1},x_i)$.  For $i=1,\dotsc, \ell$   we set
\[
S_i:=\bigl\{ (x,y)\in S;\;\;x\in [x_{i-1},x_i]\,\bigr\}.
\]
The set $S_i$ is a  union of   elementary sets 
\[
S(\beta_{i,j}, \tau_{i,j}),\;\;j=0,\dots, p_i,
\]
over the same interval $[x_{i-1},x_i]$ ,  ``stacked one above the other'', i.e.,
\begin{equation}
\beta_{i,0}(x)\leq \tau_{i,0}(x)< \beta_{i,1}(x) \leq \tau_{i,1}(x)< \cdots < \beta_{i,p_i}(x)\leq \tau_{i,p_i}(x),\;\;\forall x\in(x_{i-1},x_i).
\label{eq: orders}
\end{equation}
From  Proposition \ref{pro: ivt2} we deduce that for any $\ep$-generic $x\in (x_{i-1},x_i)$ we have $\bn(x)=p_i$. 

Both of the functions $\beta_{i,j}$ and $\tau_{i,j}$ are continuous and semialgebraic.  For any $i=1,\dotsc, \ell$, any $j=1,\dots p_i$ and any $ \hbar\in (0,\frac{\Delta}{4})$ we denote by $\gamma_{i,j}(\hbar)$ the minimum of $\beta_{i,j}-\tau_{i,j-1}$ on the interval $[x_{i-1}+\hbar, x_i-\hbar]$.  Using {\L}ojasewicz's inequality (\ref{Loja})  in the special  case $f(\hbar)=\gamma_{i,j}(\hbar)-\gamma_{i,j}(0)$ and $g(\hbar)=\hbar$ we deduce that there exists $C=C(S)>0$ and $r=r(S)\in\bZ_{>0}$ such that for any $ i=1,\dotsc, \ell$, any $j=1,\dotsc, p_i$ and any $\hbar\in (0,\frac{\Delta}{4})$ we have
\begin{equation}
\gamma_{i,j}(\hbar) > C \hbar^{r}.
\label{eq: gamma1}
\end{equation}
Fix $L>0$ and $\alpha>0$  such that for any $i=1,\dotsc, \ell$, $j=1,\dotsc, p_i$ and any $x,y\in [x_{i-1},x_i]$ we have
\[
|\beta_{i,j}(x)-\beta_{i,j}(y)|+|\tau_{i,j}(x)-\tau_{i,j}(y)|\leq L|x-y|^\alpha.
\]
Fix $\kappa_0>0$ such that $r\kappa_0<\min(1,\alpha)$.   Since 
\[
\lim_{\ve\searrow 0} \frac{\ve^{r\kappa_0}}{3\ep+L\ep^\alpha}=\infty,
\]
we  can  choose  $\ep_0>0$,  $\nu_0>0$ such that  
\begin{equation}
C\bigl(\,\nu_0\ep^{\kappa_0}\,\bigr)^{r}> 3\ep+L\ep^\alpha,\;\;\forall \ep\in (0,\ep_0].
\label{eq: gamma2}
\end{equation}
The desired  follows by  letting $\hbar= \nu_0\ve_0^{\kappa_0}$ in (\ref{eq: gamma1})  and then invoking   (\ref{eq: gamma2}) and Theorem \ref{th: sep}.
\end{proof}

\begin{definition}
The constant $\kappa_0(S)$ guaranteed by Theorem \ref{thm: noisebound} is called the \emph{separation exponent} of the set $S$.\qed
\label{def: sep constant}
\end{definition}

Theorem \ref{thm: noisebound} tells us that the jumps of $\bn_\ep$ occur within $\nu_0 \ep^{\kappa_0-1}$ pixels  from the jumps in $\bn$.  A priori, it could be possible  that, given a jumping point $x_0$ of $\bn$,  there is no jump   in $\bn_\ep$  within $\nu_0 \ep^{\kappa_0}$ pixels   of $x_0$.  Our next theorem shows that in fact this cannot happen.

\begin{theorem}
Let $S$ be a generic compact semialgebraic   set, and $\ep_0=\ep_0(S)$, $\nu_0=\nu_0(S)$ as in Theorem \ref{thm: noisebound}. Let $\kappa_0=\kappa_0(S)$ be the separation exponent of $S$. Then, there exist $\ep_1=\ep_1(S)>0$ such that  if $\ep<\min(\ep_0,\ep_1)$ and $x_0$ is a jumping point of $\bn=\bn_S$, then $\bn_\ep=\bn_{S,\ep}$ has at least one jumping point  in the interval $[x_0 - \nu_0\ep^{\kappa_0}, x_0 + \nu_0\ep^{\kappa_0}]$.
\label{th: critgen}
\end{theorem}

\begin{proof} Fix a good stratification $\eF$ of $S$  such that the function $\bh(x,y)=x$ is a stratified Morse function with respect to $(S,\eF)$.  Then there exists    exactly one critical point  $p_0\in S$  of $\bh$ such that $\bh(p_0)=x_0$.  Let $p_0=(x_0,y_0)$. 

Since $x_0$ is a jumping point of $\bn$ we have
\[
\bn(x_0^+)\neq \bn(x_0)\;\;\mbox{or}\;\;\bn(x_0)\neq \bn(x_0^-).
\]
We discuss only the case $\bn(x_0^-)\neq \bn(x_0)$ because   the other case reduces to this case applied to the region  obtained from $S$ via a reflection in the $y$-axis. For every $\ep>0$ fix an $\ep$-generic  point $x_0'(\ep)$  such that
\[
x_0'(\ep)=\begin{cases}
x_0, & \mbox{if $x_0\in\bR\setminus\ep\bZ$}\\
x_0-\frac{\ep}{2} & \mbox{if $x_0\in\ep\bZ$}.
\end{cases}
\]
 We distinguish several cases.

\medskip 

\noindent {\bf Case 1.}   $\bn(x_0^-) > \bn(x_0)$.   We can find $\delta>0$ sufficiently small   such that the interval  $(x_0 - \delta, x_0) $ will contain no  jumping points  of $S$.   The  set
\[
S_{[x_0-\delta,x_0]}:=\bigl\{ (x,y)\in S;\;\;x\in [x_0-\delta,x_0]\,\bigr\}
\]
is   a union of  elementary regions 
\[
S(\beta_j, \tau_j),\;\;j=0,\dots, m=\bn(x_0^-),
\]
``stacked one above the other'', i.e.,
\[
\beta_0(x)\leq \tau_0(x)< \beta_1(x) \leq \tau_1(x)< \cdots < \beta_{m}(x)\leq \tau_{m}(x),\;\;\forall x\in(x_0-\delta,x_0),
\]
 where $\beta_j,\tau_j$ are are continuous semialgebraic functions.  Since $\bn(x_0)<\bn(x_0^-)$ we deduce that there exist $j_0, j_1\in\{1,\dotsc, m\}$ such that $j_0\leq j_1$ and 
 \[
 \beta_{j_1}(x_0)=\tau_{j_0-1}(x_0)\;\;\mbox{and}\;\; \gamma_j:=\beta_{j}(x_0)-\tau_{j-1}(x_0)>0,\;\;\forall j\not\in [j_0, j_1].
 \]
 Thus, the elementary sets 
 \[
 S(\beta_{j_1},\tau_{j_1}),\dotsc, S(\beta_{j_0}, \tau_{j_0}), S(\beta_{j_0-1},\tau_{j_0-1} )
 \]
have a point in common, namely the critical point $p_0$; see Figure \ref{fig: jumpl}. In particular, for any $\ep>0$, the $\ep$-stacks over $x_0'(\ep)$ of these sets  also have a point in common.

 \begin{figure}[ht]
\centering{\includegraphics[height=2in,width=2.9in]{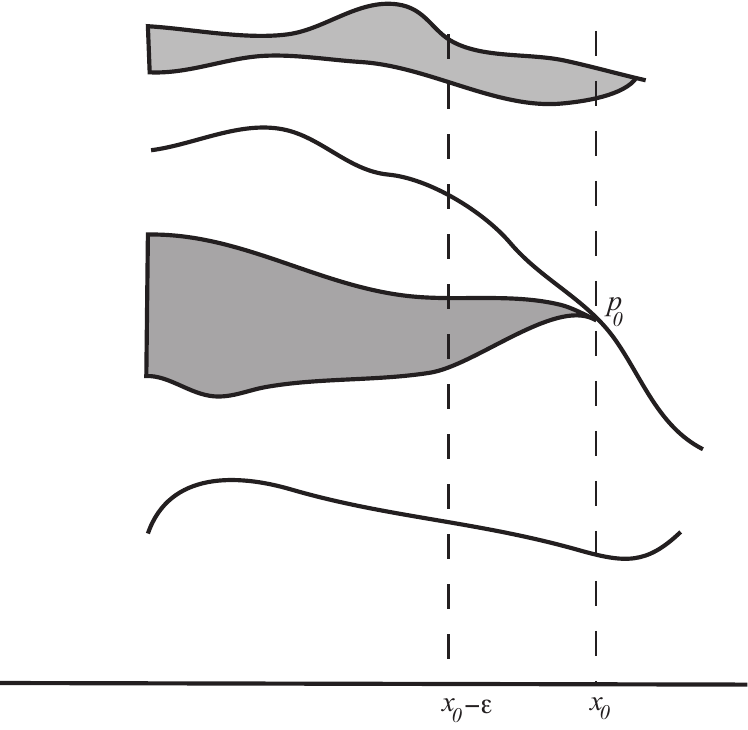}}
\caption{\sl  Near a jump point $x_0$.}
\label{fig: jumpl}
\end{figure}

 Now choose   $\ep_1$ sufficiently small so that  for $j\not\in[j_0,j_1]$  and $\ep<\ep_1$, the $\ep$-stack of $S(\beta_j,\tau_j)$ over $x_0'(\ep)$ is   disjoint  form the $\ep$-stack of $S(\beta_{j-1},\tau_{j-1})$ over $x_0'(\ep)$. Fix   $\ep<\min(\ep_0,\ep_1)$. The above discussion shows that
 \[
 \bn(x_0)=\bn_\ep(x_0'(\ep) ).
  \]
 If we set 
 \[
 x''_0(\ve): = \ve\left\lfloor\frac{x_0-\nu_0\ve^{\kappa_0}}{\ve}\right\rfloor-\frac{\ve}{2},
 \]
 then    Theorem  \ref{thm: noisebound} now implies that
 \[
\bn_\ep(\,x_0''(\ve)\,)= \bn(\,x_0''(\ve)\,)>\bn(\, x_0'(\ve)\,)=\bn_\ep(x_0'(\ep) ).
 \]
 This proves that  the interval $[x_0-\nu_0\ep^{\kappa_0}, x_0]$ contains a jumping point of $\bn_\ep$.

\medskip

 \begin{figure}[ht]
\centering{\includegraphics[height=2.2in,width=3in]{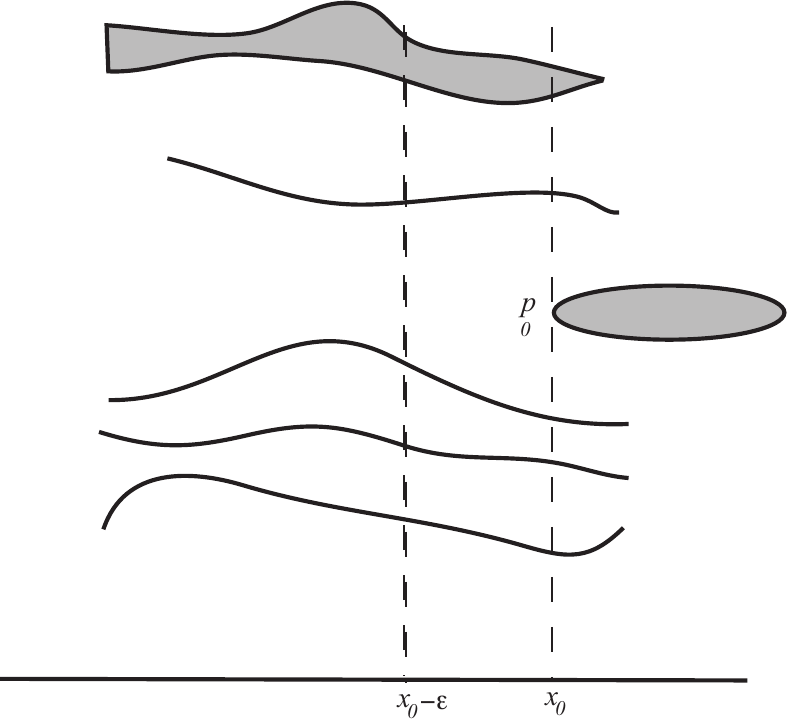}}
\caption{\sl  Near a jump point $x_0$.}
\label{fig: jumpr}
\end{figure}

\noindent {\bf Case 2.}  $\bn(x_0^-) < \bn(x_0)$.   This implies that there exist a critical point $p_0$ on the vertical line $\{x=x_0\}$ and   a tiny disk $D$ centered at $p_0$ such that  (see Figure \ref{fig: jumpr})  
\begin{equation}
D\cap \bigl\{ (x,y)\in S;\;\;x<x_0\,\bigr\}=\emptyset.
\label{eq: isolate}
\end{equation}
To see why this is the case note that the condition  $\bn(x_0^-) < \bn(x_0)$ implies that a component $K$ of $S\cap\{x=x_0\}$ is disjoint from  the closure of  $S\cap \{x_0-\ve\leq  x < x_0\}$ for $\ve$ sufficiently small.   The  component $K$ is either a point, or a nontrivial compact interval.   The  second possibility is prohibited  by the genericity of $S$ because any point of $K$ is a stratified critical point of the projection of $S$ onto the $x$-axis. This $K$ consists of a single point $p_0$ which is critical and, by construction,  it satisfies (\ref{eq: isolate}).

In  particular, this shows that $p_0$ is an isolated point of the set
\[
S_{x\leq x_0}=\bigl\{(x,y)\in S;\;\;x\leq x_0\,\bigr\}.
\]
  If $\bn(x_0^-)=0$, the conclusion is obvious.  We assume that $\bn(x_0^-)>0$.     Choose $\delta>0$ such that the interval $[x_0-\delta,x_0)$ contains no jumping point of $S$. Set
\[
R:= \cl\bigl(\,S_{[x_0-\delta,x_0]}\setminus \{p_0\}\,\bigr).
\]
Then $R$ is a  union of   simple regions
\[
S(\beta_j, \tau_j),\;\;j=0,1\dotsc, m=\bn(x_0^-)-1,
\]
where $\beta_j$ and $\tau_j$ are continuous semialgebraic functions functions such that
\[
\beta_0(x)\leq \tau_0(x) <\beta_1(x)\leq \tau_1(x)<\cdots <\beta_m(x)\leq \tau_m(x),\;\;\forall x\in [x_0-\delta,x_0].
\]
We can find  $\ep_1=\ep_1(S)$ such that for any $\ep<\ep_1$ and any  $\ep$-generic $x\in [x_0-\delta,x_0]$ we have:
\begin{itemize}
\item $\bn_{R,\ep}(x)=\bn_S(x)=m+1=\bn_S(x_0^-)$, and

\item  the $\ep$-column of $S_{[x_0-\delta,x_0]}$  over  $x_0$  consists of $\bn(x_0)=m+2$ stacks. 

\end{itemize}

Theorem \ref{thm: noisebound} implies that
\[
\bn_{S,\ep}(x)=\bn_S(x)=\bn_S(x_0^-)=m+1\;\; \forall x\in [x_0-\delta,x_0-\nu_0\ep^{\kappa_0}]\setminus \bZ\ep.
\]
On the other hand, $\bn_{S,\ep}(x_0^-)=m+2$.    Thus the interval $[x_0-\nu_0\ep^{\kappa_0},x_0]$ must contain a jumping point of $\bn_{S,\ep}$.

\end{proof}

\begin{remark}    Theorem  \ref{thm: noisebound}   states that the two functions  $\bn$ and $\bn_\ep$ coincide  at points situated at a distance at least $\nu_0(S)\ep^{\kappa_0 - 1}$ pixels away from  the jumping points of $\bn$. On the other hand, Theorem \ref{th: critgen} shows that, for a generic semialgebraic set, then  within $\nu_0(S)\ep^{\kappa_0 - 1}$ pixels  from a jumping point of $\bn$ there must be jumping points of $\bn_\ep$.\qed
\label{rem: sepres}
\end{remark}

\begin{definition} Let $S$ be a generic semialgebraic set in $\bR^2$ and the constants $\ep_0(S)$ and $\ep_1(S)$ as defined in Theorems \ref{thm: noisebound} and  \ref{th: critgen}. We set
\[
\hbar(S):=\min\bigl(\, \ep_0(S),\;\ep_1(S)\,\bigr),
\]
and  we will refer to it as the \emph{critical resolution} of $S$.

\qed
\label{def: constants}
\end{definition}

\section{Approximation of generic semi-algebraic sets}
\label{s: 4}
\setcounter{equation}{0}

This section is the heart of the paper.  Here we will describe an algorithm which will approximate a generic semi-algebraic set using only its pixelations, and then prove a very strong convergence result for this approximation.  This algorithm is based on the central algorithm of \cite{Row}, updated to handle the additional  complexities of semi-algebraic sets.

We first observe that when   narrow vertical strips around the  jumping set $\eJ_S$ are removed from a semi-algebraic set $S$,  the remainder is a disjoint union of elementary sets.  Corollary \ref{cor: piece curv} indicates a good way to approximate continuous semi-algebraic functions, and Theorems \ref{thm: noisebound} and \ref{th: critgen} indicate that for small $\ep$, the jumping points of $S$ become close to the jumping points of $P_\ep(S)$.  Therefore a viable approximation technique is to treat parts of $P_\ep(S)$ which occur near jumping points as noise (to be approximated crudely) and to approximate outside of this noise by means of Corollary \ref{cor: piece curv}.

There are two quantities which must be used  in this approximation.  The first is the previously mentioned spread function $\sigma$ which determines the width of line segments to be used in approximating outside of noise.  From Corollary \ref{cor: piece curv} we know that this spread function should satisfy the following limits:
\[
\lim_{\ep \searrow 0} \ep \sigma(\ep) = 0, \;\; \lim_{\ep \searrow 0} \ep (\sigma(\ep))^2 = \infty.
\]
The second quantity determines the width of the noise, measured in pixels, about jumping points.  We will call this quantity $\nu$ and refer to it as the \emph{noise width}.  It is defined as follows:
\begin{definition}
Let $S$ be a semi-algebraic set and $\kappa_0$ be its separation exponent.  Then a \emph{noise width} $\nu$ is a function $\nu: \bR^+ \to \bZ^+$ which satisfies the following equations:
\begin{subequations}
\begin{equation}
\lim_{\ep \searrow 0} \ep \nu(\ep) = 0,
\label{eq: smallnoise}
\end{equation}
\begin{equation}
\lim_{\ep \searrow 0} \frac{\ep \nu(\ep)} { \ep^{\kappa_0}} = \infty,
\label{eq: suff-noise}
\end{equation}
\end{subequations}
\end{definition}
The first property in this definition ensures that noise is a highly localized phenomenon.  The second property implies that $\ep \nu(\ep)$ increases faster than $\ep^{\kappa_0}$ so that the noise will eventually contain all fake cycles (consult Remark \ref{rem: sepres}). For a reasonable approximation we must have an a priori  estimate of $\kappa_0$.  The choice  $\kappa_0 = \frac{1}{2}$ works for many $S$. We could then set $\nu(\ep) = \lceil \ep^{\frac{-2}{3}}\rceil$.

\begin{definition}
Let $S$ be a semialgebraic set and $P_\ep(S)$ its $\ep$-pixelation.  If $A \subset \bR$ then the \emph{part of $P_\ep$ over $A$} is the set
\[
P_\ep(S) \cap (A \times \bR)
\]
\label{def: over}\qed
\end{definition}

\begin{algo}

\begin{enumerate}
\item Choose a spread $\sigma$ such that $\ep \sigma(\ep)^2 \to \infty$ and $\ep \sigma(\ep) \to 0$ as $\ep \to 0$.

\item Choose a separation exponent $\kappa_0>0$  and noise width $\nu=\nu(\ve)$ such that $\frac{\ep \nu(\ep)}{\ep^{\kappa_0}} \to \infty$ and $\ep \nu(\ep) \to 0$ as $\ep \to 0$.

\item  For each point $p \in \eJ_{S,\ep}$ set
\[
\zeta_\ep^-(p) :=-\ve\nu(\ve)+\ve\left\lfloor\frac{p}{\ve}\right\rfloor-\frac{\ve}{2},\;\;\zeta_\ep^+(p):= \ve\nu(\ve)+\ve\left\lceil\frac{p}{\ve}\right\rceil+\frac{\ve}{2},
\]
\[
\Delta_\ep(p):=[\zeta^-_\ep(p), \zeta^+_\ep(p)].
\]
The set 
\[
\Delta_\ep := \bigcup_{p \in \bsJ_\ep} \Delta_\ep(p).
\]
is called the \emph{noise} set of $P_\ep(S)$ and its connected components are called the \emph{noise intervals} of $P_\ep(S)$.

\item Define $\eR_\ep$ to be the closure of  $\bR\setminus \Delta_\ep$. We call $\eR_\ep$ the \emph{regular} set of $P_\ep(S)$, and its connected components are called \emph{regular intervals} of $P_\ep(S)$.

\item For each bounded regular interval $I \subset \eR_\ep$   and each connected component $\eC$  of $P_\ep(S) \cap (I\times \bR)$, the part of $P_\ep(S)$ over  the regular regular intervals $I$, do the following:

\begin{enumerate} 

\item Choose compatible upper and lower samples $\Xi_\ep^+$ and $\Xi_\ep^-$ with spread $\sigma$.

\item Generate the PL-approximation determined by the above upper and lower samples. 

\item The union of all PL-approximations found in the above step is called the \emph{regular approximation} which we denote by $S_\ep^{\; reg}$.

\end{enumerate}

\item For each noise interval $I \subset \Delta_\ep$ denote by $\eC_\ep(I)$ the set of connected components of $ P_\ep(S)$ over $I$. 

\begin{enumerate}
 
\item  For every  $C\in \eC_\ep(I)$ we denote by $U_C$ (resp. $L_C$)  the  highest (resp. lowest)   $y$-coordinate of  the center of a pixel  in $C$.

  \item Denote by  $\eP_\ep(C)$ the  rectangle $I\times [L_C,U_C]$; see Figure \ref{fig: noise blocks}.
 
 \begin{figure}[ht]
\centering{\includegraphics[height=1.5in,width=1.5in]{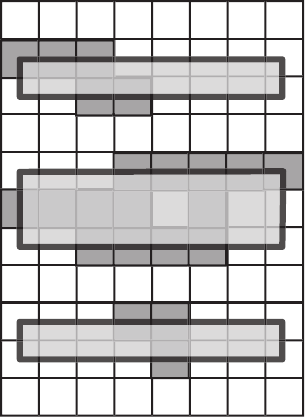}}
\caption{\sl  Covering up the noise.}
\label{fig: noise blocks}
\end{figure}

\item The \emph{noise approximation over I}, which we denote by $P_\ep(I)$, is the union
\[
P_\ep(I):=\bigcup_{C \in C_\ep(I)} \eP_\ep(C)
\]

 \end{enumerate}
 \item The union of all $P_\ep(I)$, where $I$ are the noise intervals in $\Delta_\ep$, is called the \emph{noise approximation} which we denote by $S_\ep^{\; noise}$.
\item The final approximation $S_\ep$ is simply the union of the noise and regular approximations, i.e.
\[
S_\ep =  S^{\; noise}_\ep \cup S_\ep^{\; reg}
\]
\end{enumerate}
This final set $S_\ep$ will be piecewise linear by construction.\qed
\label{alg: process}
\end{algo}

\begin{ex} If we apply the above algorithm applied to  the $\ve$-pixelation of the unit circle $\{x^2+y^2=1\}$, where $\ve\approx \frac{1}{16}$,   we obtain  the region depicted  in Figure \ref{fig: circle}. The noise blocks are the two rectangles that cover the noise region.
\begin{figure}[ht]
\centering{\includegraphics[height=1.5in,width=1.5in]{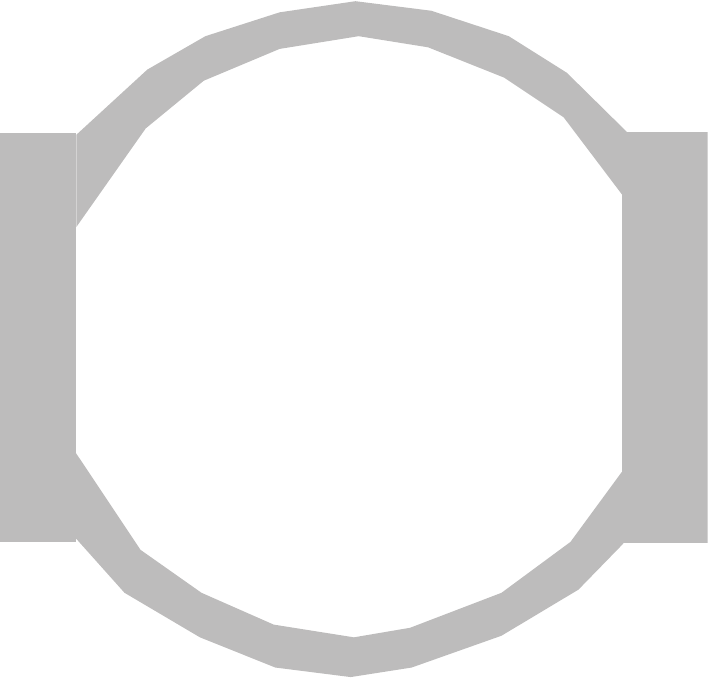}}
\caption{\sl  Recovering a circle from a rough pixelation.}
\label{fig: circle}
\end{figure}

We  can ``beautify'' a bit the final product by running  the algorithm    on  the pixelation obtained by a  ninety degree rotation, i.e., by reversing the roles of $x$ and $y$-axes.  We obtain two $PL$ approximations.   The intersection of the two is another $PL$ approximation with smaller  noise blocks.   The  final  product is a  region  which  closely resemble  an annular region of width $\approx \ve$. The unit circle is the ``median'' circle of this annular region.\qed
\end{ex}

The approximation $S_\ve$ produced by the above algorithm is ``good",   meaning  that it captures both topological and geometric information such as area, perimeter, and curvature measures.  The precise notion  of ``good approximation'' relies on the concept of  \emph{normal cycle}.

The normal cycle is a correspondence that associates to each compact planar semialgebraic  set $X$ a $1$-dimensional current   $\bsN^X$        on the unit sphere  tangent bundle of $\bR^2$ \label{p: normal},
\[
\bsS(T\bR^2)= \bigl\{ (\bv, \bp)\in \bR^2\times \bR^2;\;\; |\bv|=1\,\bigl\}.
\]
For a precise definition of this object we refer to \cite{Ber, Fu2, Mor, LC}.        Here we  will content   ourselves with a  brief description of its construction.

For a  semialgebraic compact domain $D\subset \bR^2$ with $C^2$-boundary the normal cycle $\bsN^D$ has a simple description. It is the current of integration   given by the   closed curve  $\eG_D\subset \bsS(T\bR^2)$ 
\[
\eG_ D = \bigl\{ (\bn(\bp), \bp)\in \bsS(T\bR^2);\;\;\bp\in \pa D\,\bigl\}
\]
where  $\bn(\bp)$ denotes the unit  outer normal to $\pa D$ at $\bp\in \pa D$.    Equivalently, $\eG_D$ is the graph the Gauss map
\[
\pa D\ni \bp \mapsto \bn(\bp)\in S^1.
\]
Clearly, in this case,  the normal cycle contains all the curvature information   concerning the boundary of $D$.

More generally, if $S$ is a compact semialgebraic set, the we can find a $C^3$, proper semialgebraic function $f: \bR^2\ra [0,\infty)$ such that $S= f^{-1}(0)$.   For all $\ep>0$ sufficiently small the region $S_\ep:=\{f\leq \ep\}$ is a compact semialgebraic domain with $C^2$-boundary  so we can define the normal cycle $\bsN^{S_\ep}$ as above. One can show that as $\ep \ra 0$  the currents   $\bsN^{S_\ep}$ converge  weakly to a current  which by definition  is the normal cycle of $S$. The hard part is to prove that this current is independent of the choice of defining function $f$.    This current is   a current of integration along a finite number of oriented semialgebraic arcs  in $\bsS(T\bR^2)$.  We refer to \cite{Mor} to a more in depth description  of  the normal cycle of planar semialgebraic sets.  In particular, in \cite{Mor}  one can  see how this current  captures   the various curvature properties of $S$.

The following is the main result of this paper.

\begin{theorem}
Let $S$ be a generic compact semi-algebraic subset of the plane.  For each $\ep$, let $S_\ep$ be the PL set constructed using the Algorithm \ref{alg: process}. Denote by $\bsN^S$ (resp. $\bsN^{S_\ve}$) the normal cycle of $S$ (resp. $S_\ve$). 
Then $\bsN^{S_\ve}$ converges to $\bsN^{S}$ weakly and in the flat topology as $\ve\to 0$.
\label{th: main}
\end{theorem}

\begin{proof} First we note that the approximation converges in the Hausdorff metric to the original set.  This is because each vertex of a line segment is taken from a pixel which contains a piece of the boundary of the original set.  Since every pixel of the pixelation can be at most $\ep \sqrt{2}$ far from the original set, this forces the approximation into a tube around the original set which becomes arbitrarily small as $\ep$ goes to $0$.

From here the strategy of the proof will make heavy use of the inclusion-exclusion principle  satisfied  by the normal cycle correspondence $X\mapsto \bsN^X$. More precisely, this means that for any compact semi-algebraic sets $X$ and $Y$, we have
\[
\bsN^{X\cup Y} =\bsN^{X}+\bsN^{Y}-\bsN^{X\cap Y}.
\]
We will use this principle to reduce the calculation of the normal cycle of $S_\ep$ to calculations of the normal cycle of  simpler subsets of $S_\ep$.


First, we need to introduce   some more notation.  We set $S_0 = S$.  For each $\ep > 0$ and each $c \in \eJ_{S,\ep}$ we indicate by $\eI_\ep(c)$ the $\ep$-noise interval containing $c$.   For each $c \in \eJ_S$ we set $\eI_0(c) := \{c\}$.  We then define, for each $\ep \geq 0$, the noise strip $\eN_\ep(c)$ as
\[
\eN_\ep(c) := \{(x,y); x \in \eI_\ep(c)\}
\]
and set
\[
\eN_\ep := \bigcup_{c \in \eJ_{S,\ep}} \eN_\ep(c),\;\;\hat{\eR}_\ep:=\bR^2 \setminus \eN_\ep.
\]
For each $\ep \geq 0$ we construct a graph $\Gamma_\ep$ as follows.  The vertex set is the set of connected components of $\eN_\ep \cap S_\ep$.  The edge set is the set of connected components of $\hat{\eR}_\ep \cap S_\ep$, so that two vertices $v_1,v_2$ are connected by an edge  if and only if there is a component of $\hat{\eR}_\ep \cap S_\ep$  whose closure  intersects  the two  components  of $\eN_\ep \cap S_\ep$ defining $v_1,v_2$.  This graph is the Reeb graph   of the  projection of $S_\ep$ onto the $x$-axis. (We refer to \cite[VI.3]{EH} for a definition of the Reeb graph.)

 Observe that  there is a $\delta > 0$ such that for all $\ep \in (0,\delta]$ the graph $\Gamma_\ep$ is isomorphic to the graph $\Gamma_0$.  Let 
\[
\ep_2 := \min \{\hbar(S), \delta\},
\]
 where $\hbar(S)$ is the critical resolution defined in Definition \ref{def: constants}.  For the remainder of the proof we will deal only with $\ep \in [0,\ep_2]$.


Let $\eV_\ep$ be the set of vertices of $\Gamma_\ep$ and $\eE_{\ep}$ be the set of edges of $\Gamma_\ep$.  Note that, by the above equivalence of Reeb graphs, for $\ep \in [0,\ep_2]$, there is a natural bijection between the vertices of $\Gamma_\ep$ with those of $\Gamma_0$ and similarly for the edges.    Given a vertex $\bv$ of $\Gamma_0$  we set
\[
 E_\bv:=\mbox{ the set of edges of $\Gamma_0$ incident to $\bv$.}
 \] 
For any vertex $\bv$ of $\eV_0$ and $\ep \in [0,\ep_2]$ we indicate by $C_{\bv,\ep}$ the connected component of $\eN_\ep \cap S_\ep$ corresponding to the vertex.  Similarly, for any edge $\be$ of $\Gamma_0$ we indicate by $C_{\be,\ep}$ the \emph{closure} of the connected component of $\hat{\eR}_\ep \cap S_\ep$ corresponding to $\be$.   We have  the following result, (compare \cite[Lemma 5.3]{Row}).

\begin{lemma}  For any $\ve\in [0,\ep_2]$ we  have
\begin{equation}
\bsN^{S_{\be}} = \sum_{\bv \in \eV_0} \bsN^{C_{\bv,\ep}} + \sum_{\be \in \eE_0} \bsN^{C_{\be,\ep}} - \sum_{\bv \in \eV_0} \sum_{\be \in \eE_0}\bsN^{C_{\bv,\ep} \cap C_{\be,\ep}}.
\label{eq: key-normal}
\end{equation}
\end{lemma}

\begin{proof} Note that we have  a  decomposition
\begin{equation}
S_\ep=\left(\bigcup_{\bv\in\eV_0} C_{\bv, \ep}\right)\cup \left(\bigcup_{\bse\in \eE_0} C_{\bse,\ep}\right).
\label{eq:  decomp}
\end{equation}
We need to discuss separately the cases $\ve>0$ and $\ve=0$.

\smallskip

\noindent {\bf 1.} Assume that $\ep\in (0,\ep_2]$. In this case  we have

\begin{equation}
C_{\bv,\ep}\cap C_{\bv',\ep}=\emptyset=C_{\bse,\ep}\cap C_{\bse',\ep},\;\;\forall \bv\neq \bv',\;\;\bse\neq\bse'.
\label{eq: disj1}
\end{equation}
The equality (\ref{eq: key-normal})  now follows from the inclusion-exclusion principle  applied to the decomposition (\ref{eq: decomp})   satisfying the overlap conditions (\ref{eq: disj1}).

\smallskip

\noindent {\bf 2.}  Assume that $\ep=0$. In this case the overlap conditions are more complicated.  We have
\begin{subequations}
\begin{equation}
C_{\bv,0}\cap C_{\bv',0}=\emptyset,\;\;\forall \bv\neq\bv',
\label{eq: disj2a}
\end{equation}
\begin{equation}
C_{\bse,0}\cap C_{\bse',0}=\emptyset\Llra \bse\cap\bse'=\emptyset,
\label{eq: disj2b}
\end{equation}
\end{subequations}
where the condition $\bse\cap\bse'=\emptyset$ signifies that the edges $\bse$ and $\bse'$ have no vertex in common. Recall that $E_\bv$ denotes the set of edges of $\Gamma_0$ incident to the vertex $\bv$.   We have  
\begin{equation}
\bigcap_{\bse\in A} C_{\bse, 0} = C_{\bv,0},\;\;\forall \bv\in\eV_0,\;\; \emptyset\neq A\subset E_{\bv}.
\label{eq: disj3}
\end{equation}
Using (\ref{eq: decomp}), (\ref{eq: disj2a}),  (\ref{eq: disj2b}),  (\ref{eq: disj3}) and the inclusion-exclusion principle we deduce
\[
\begin{split}
\bsN^S = &\sum_{\bv\in \eV_0} \bsN^{C_{\bv,0}}+\sum_{\bse\in \eE_0} \bsN^{C_{\bse,0}}-\sum_{\bv \in \eV_0}\sum_{\bse\in E_\bv} \bsN^{C_{\bv,0}\cap C_{\bse,0}}\\
&+ \sum_{\bv\in \eV_0} \sum_{ \emptyset\neq A\subset E_{\bv}} (-1)^{|A|+1} \bsN^{ C_{\bv}\cap (\bigcap_{\bse\in A}) C_{\bse,0} } + \sum_{\bv\in \eV_0} \sum_{\emptyset\neq  A\subset E_{\bv}} (-1)^{|A|} \bsN^{ \bigcap_{\bse\in A} C_{\bse,0} }
\end{split}
\]
\[
\begin{split}
&=\sum_{\bv\in \eV_0} \bsN^{C_{\bv,0}}+\sum_{\bse\in \eE_0} \bsN^{C_{\bse,0}}-\sum_{\bv \in \eV_0}\sum_{\bse\in E_\bv} \bsN^{C_{\bv,0}\cap C_{\bse,0}}\\
&+\sum_{\bv\in \eV_0}\, \underbrace{\left( \sum_{\emptyset\neq  A\subset E_\bv} \bigl(\,(-1)^{|A|+1}+  (-1)^{|A|}\,\bigr)\,\right)}_{=0}\,\bsN^{C_{\bv,0}}
\end{split}
\]
\end{proof}

The above lemma shows  that Theorem \ref{th: main} will follow  once we prove that the  three equalities below are satisfied for every edge $\be$ and vertex $\bv$ in $\Gamma_0$.

\begin{subequations}
\begin{equation}
\lim_{\ep \searrow 0} \bsN^{C_{\bv,\ep}} = N^{C_{\bv,0}}
\label{eq: limnoise}
\end{equation}
\begin{equation}
\lim_{\ep \searrow 0} \bsN^{C_{\bv,\ep} \cap C_{\be, \ep}} = \bsN^{C_{\bv,0} \cap C_{\be,0}}
\label{eq: limint}
\end{equation}
\begin{equation}
\lim_{\ep \searrow 0} \bsN^{C_{\be,\ep}} = \bsN^{C_{\be,0}},
\label{eq: limreg}
\end{equation}
\end{subequations}
where the convergence in each limit is meant in the weak sense of currents.

Each of these equations will rely on an approximation result of normal cycles proved by Joseph Fu in \cite{Fu1}.  A restricted version of this theorem, which shall suffice for the purposes of this paper, is stated below.

\begin{theorem}[Approximation Theorem] Suppose $S$ is a  compact semialgebraic  subset of the plane and for each $\ep>0$ we are given a  compact semialgebraic  subset  $S_\ep$  of the plane with the following properties.

\begin{enumerate}
\item There is a compact set $K \subset \bR^2$ which contains each $S_\ep$.
\item There is a $M \in \bR$  such that
\[
{\rm mass}\bigl(\,\bsN^{S_\ep}\,\bigr)\leq M,\;\;\forall \ep.
\]
\item For almost every $\xi \in \Hom(\bR^2,\bR)$ and almost every $c \in \bR$ we have
\[
\lim_{\ep \searrow 0} \chi(S_\ep \cap \{\xi \geq c\}) = \chi(S \cap \{\xi \geq c\})
\]
\end{enumerate}
Then $\bsN^{S_\ep}$  converges to $\bsN^S$ as $\ep\ra 0$ weakly and in the flat metric.\qed
\label{thm: fu}
\end{theorem}

Equations (\ref{eq: limnoise}) and (\ref{eq: limint}) will both follow from applying this approximation theorem to the case of rectangles.  Therefore, the following lemma will be useful.

\begin{lemma}
Suppose $(S_\ep)_{\ep>0}$ is a family of compact convex polygons in the plane that converge in the Hausdorff  metric to a compact convex polygon $S$ as $\ve\searrow 0$. Then  $\bsN^{S_\ep}$ converges weakly to $\bsN^S$ as $\ep\searrow 0$.
\label{lem: quad}
\end{lemma}

\begin{proof}
We argue by proving the conditions of Fu's Theorem.   Observe first that there exists $R>0$ such that
\[
\dist(S_\ep, S)<R,\;\;\forall \ep
\]
 and thus  the condition (1) of the Approximation Theorem.   The  computations of \cite[Chap. 23]{Mor}  show that the mass of the normal cycle of a convex polygon $P$ is  equal to $2\pi +{\rm length}\,(P)$.    From  Hadwiger's characterization theorem \cite[Thm. 9.1.1]{KR}  we deduce that
 \[
 \lim_{\ep\ra 0} {\rm length}\,(S_\ep)={\rm length}\,(S)
 \]
 and  thus condition (2) is also satisfied.

Therefore we must show that for almost every $\xi \in \Hom(\bR^2,\bR)$ and almost every $c \in \bR$ we have
\[
\lim_{\ep \searrow 0} \chi(S_\ep \cap \{\xi \le c\}) = \chi(S \cap \{\xi \le c \})
\]
Note that $S$ and each $S_\ep$ are all convex subsets of the plane.  Therefore any intersection with a half-plane is either empty, or  a contractible set.  Therefore to prove the convergence of Euler characteristic on half-planes we need only prove that a half plane  $H$ will only intersect $S_\ep$ for small $\ep$ if and only if it intersects $S$.    This is true since $H\cap S_\ep$ converges in the Hausdorff metric to $H\cap S$.
\end{proof}

\noindent {\bf Proof of (\ref{eq: limnoise})} Fix a vertex $\bv \in \eV_0$.  The set $C_{\bv,0}$ is a subset of a vertical line over a jumping point, and  so it is either a point or a line segment.  For every $\ep \in [0,\ep_2]$, the set $C_{\bv,\ep}$ is a rectangle which spans a noise interval $I_\ep$ and contains $C_{\bv,0}$.  The width of $I_\ep$ (and so the width of $C_{\bv,\ep}$) is proportional to $\ep \nu(\ep)$, and so vanishes as $\ep \to 0$ (by choice of $\nu$).

The rectangle $C_{\bv,\ep}$ is constructed by choosing the highest and lowest pixels from the component of $I_\ep\cap P_\ep(S)$ containing $C_{\bv,0}$.  For sufficiently small $\ep$, the noise interval $I_\ep$ will be thin enough so that $C_{\bv,\ep} \cap S$ can be described as a number of regions lying between the graphs of functions which are $C^2$ everywhere except possibly at the jumping point.  This implies that for small $\ep$, the height of $C_{\bv,\ep}$ differs from the height of $C_{\bv,0}$ be an arbitrarily small amount.

Since the height and width of $C_{\bv,\ep}$ converge to the height and width of $C_{\bv,0}$ and since each $C_{\bv,\ep}$ contains $C_{\bv,0}$, the rectangles $C_{\bv,\ep}$ converge in the Hausdorff metric to $C_{\bv,0}$. Therefore by Lemma \ref{lem: quad} $\lim_{\ep \searrow 0} \bsN^{C_{\bv,\ep}} = \bsN^{C_{\bv,0}}$.

\smallskip

\noindent {\bf Proof of (\ref{eq: limint})} Fix an edge $\be \in \eE_0$ and a vertex $\bv \in \eV_0$.  If $C_{\be,0} \cap C_{\bv,0} = \emptyset$, then for sufficiently small $\ep$ the component $C_{\be,\ep}$ will also not intersect $C_{\bv,\ep}$ and so the convergence in normal cycle follows.

If the sets  $C_{\be,0}$ and $C_{\bv,0}$ do in fact intersect, then note that $C_{\be,0} \cap C_{\bv,0} = C_{\bv,0}$ (since the vertex is a connected component over a point, and the edges is a connected component over an interval which overlaps that point).

The intersection $C_{\be,\ep} \cap C_{\bv,\ep}$ is a vertical line segment.  In fact it is either the right or left edge of $C_{\bv,\ep}$.  However, since $C_{\be,\ep}$ converges to a vertical line segment $C_{\be,0}$, it follows that its left right and right edges converge to the same line segment.  Therefore  (\ref{eq: limint}) follows from (\ref{eq: limnoise}).

\smallskip

\noindent {\bf Proof of (\ref{eq: limreg})} We again plan to use the Approximation Theorem. The condition (1) of Theorem \ref{thm: fu} is plainly satisfied while condition (2) follows from  Corollary \ref{cor: piece curv} and   the   explicit description of the mass of the normal cycle  of a planar set given in \cite[Chap. 23]{Mor}.    All that is left to do is to verify condition (3)  of the Approximation Theorem.

  The component   $C_{\bse, 0}$ is  an elementary region  defined  by continuous  semialgebraic functions 
  \[
  \beta_\bse,\tau_\bse:[a,b]\ra \bR.
  \]
  More precisely, this means  that 
\[
\beta_\bse(x)\leq \tau_\bse(x),\;\; \forall x\in [a,b],
\]
and
\[
C_{\bse,0}=\bigl\{ (x,y)\in \bR^2;\;\;x\in [a,b],\;\; \beta_\bse(x)\leq y\leq \tau_\bse(x)\,\bigr\}.
\]
There exists an integer $n>0$ and points 
\[
a=c_0<c_1<\cdots <c_n=b
\]
 such that for any  $i=1,\dotsc, n$ the restrictions of $\beta_\bse$ and $\tau_\bse$ to $(c_{i-1}, c_i)$ are real analytic. Moreover, since the set $S$ is generic,  the   derivatives $\beta_\bse'$ and $\tau_\bse'$ are bounded   near $c_1,\dotsc, c_{n-1}$. In particular,  the functions $\beta_\bse$ and $\tau_\bse$ are  locally Lipschitz  on  the open interval $(a,b)$. We will refer to the points
 \[
 \bigl(c_j,\beta_\bse(c_j)\,\bigr),\;\;\bigl(c_j,\tau_\bse(c_j)\,\bigr),\;\;j=0,1,\dotsc, n,
 \]
 as the \emph{vertices} of $C_{\bse,0}$.  The other points on these graphs are called \emph{regular}. Now fix a  constant $c\in\bR$ and a linear map $\xi: \bR^2\ra \bR$, $\xi(x,y)=ux+vy$, with the following generic properties:
 
 \begin{itemize}
 
 \item[$\mathbf{G_1}.$] The restriction of $\xi$ to $C_{\bse,0}$ is a stratified Morse function, and $v\neq 0$, i.e., the level sets of $\xi$ are not vertical lines.
 
 \item[$\mathbf{G_2}.$] The constant $c$ is not a critical value of $\xi|_{C_{\bse,0}}$.
 
 \item[$\mathbf{G_3}.$] The line $L_{\xi,c}:=\{\xi=c\}$ does not contain any of the vertices of $C_{\bse, 0}$.
 
 \end{itemize}
 
 We will show that
 \begin{equation}
 \lim_{\ep \searrow 0}\chi\Bigl(\, C_{\bse,\ep}\cap \{\xi\geq c\,\}\,\Bigr)=\chi\Bigl(\, C_{\bse,0}\cap \{\xi\geq c\,\}\,\Bigr).
 \label{eq: final-conv1}
 \end{equation}
 For $\ve>0$  sufficiently small we set
 \begin{equation}
 a_\ep:=a+\nu(\ep)\ep,\;\; b_\ve:=b-\nu(\ep)\ep,\;\;C_{\bse,0}^\ve=C_{\bse ,0}\cap \bigl([a_\ve,b_\ve]\times \bR\,\bigr).
 \label{eq: cve0}
 \end{equation}
 Above we assume that $\ep$ is small enough so that $a_\ep<b_\ep$.

 Let us observe that the conditions $\mathbf{G_1}$,$\mathbf{G_2}$ and $\mathbf{G_3}$ imply that for  $\ep$ sufficiently small we have 
 \[
 \chi\Bigl(\, C_{\bse,0}\cap \{\xi\geq c\,\}\,\Bigr)=\chi\Bigl(\, C_{\bse,0}^\ep\cap \{\xi\geq c\,\}\,\Bigr).
 \]
 Thus, to prove  (\ref{eq: final-conv1}) it suffices to show that
 \begin{equation}
 \chi\Bigl(\, C_{\bse,\ep}\cap \{\xi\geq c\,\}\,\Bigr)=\chi\Bigl(\, C_{\bse,0}^\ep\cap \{\xi\geq c\,\}\,\Bigr),\;\;\forall \ep\ll 1.
 \label{eq: final-conv2}
 \end{equation}

 The region $C^\ve_{\bse,0}$ is an elementary region  defined by the $PL$ functions
 \[
 \beta_{\bse,\ep},\tau_{\bse,\ep}: [a_\ep,b_\ep]\ra \bR,\;\;\beta_{\bse,\ep}(x)\leq\tau_{\bse,\ep}(x),\;\;\forall  x\in [a_\ep, b_\ep].
 \]
It is the part of the component  $C_{\be,0}$  outside the $\ve$-noise strips.

We need to develop some terminology to handle the intersections of these $PL$ boundary functions.  If $f:[s,t]\ra \bR$ is a piecewise $C^2$ function, then we say  that the line \emph{$L_{\xi,c}$ intersects the graph  of $f$  transversally} at a point $\bp_0= (x_0,f(x_0))$ if  there exists a $\delta>0$ such that  the function 
 \[
 x\mapsto  \xi(x,f(x))
 \]
 is  differentiable  on the set $0<|x-x_0| \leq \delta$  and its derivative has  constant  sign on this set.   We will denote by $\sign(\bp_0, f)\in \{\pm 1\}$ this sign. Thus, if $\sign(\bp_0, f)=1$, then the  curve  
 \[
 x\mapsto (x,f(x)), \;\;|x-x_0|\leq \delta, 
 \]
  intersects  the line $L_{\xi,c}$ at $\bp_0$ coming from the half-plane $\{\xi<c\}$ and entering the half-plane $\{\xi>c\}$.
  
  For any point $\bp\in\bR^2$ and any $r>0$ we denote by $\Sigma_r(\bp)$ the closed square
  \[
  \Sigma_r(\bp):=\bigl\{\, (x,y)\in\bR^2;\;\; |x-x(\bp)|\leq r,\;\;|y-y(\bp)|\leq r\,\bigr\}.
  \]
 We need to discuss separately three cases.

 \medskip
 
  \begin{figure}[ht]
\centering{\includegraphics[height=2in,width=3in]{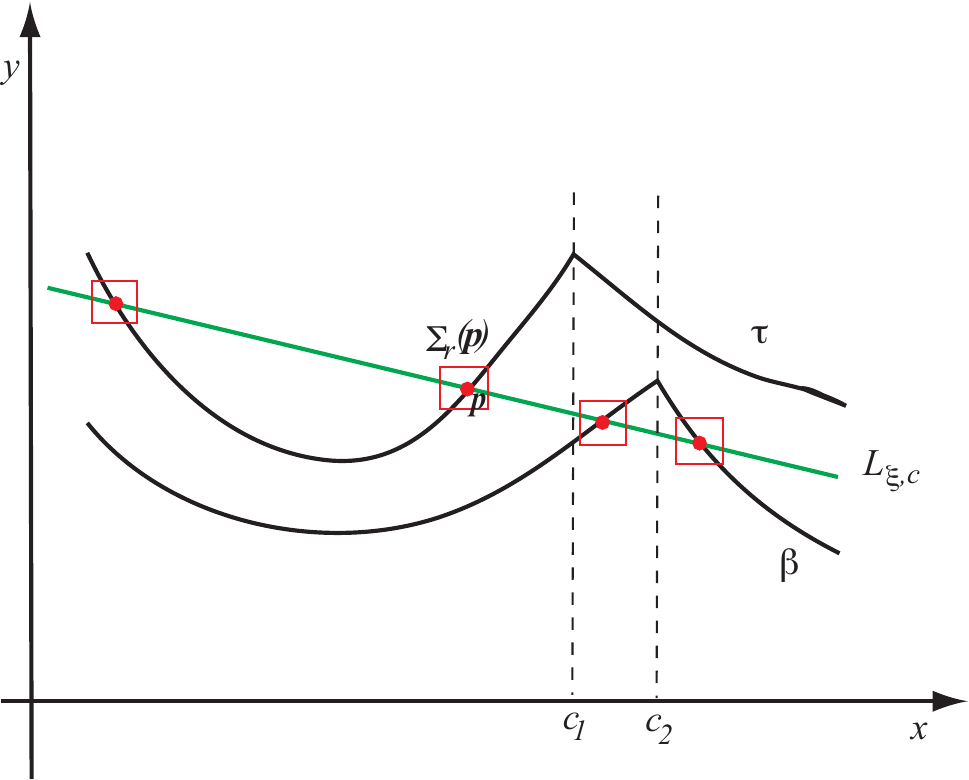}}
\caption{\sl Isolating the intersection points of $L_{\xi,c}$ with the graphs of $\beta$ and $\tau$ so that the squares $\Sigma_r(\bp)$ do not touch any of the vertical lines containing  the  singular points of these graphs.}
\label{fig: isolate}
\end{figure}
 
 \noindent {\bf Case 1.} \emph{The elementary set $C_{\bse,0}$ is nondegenerate}, i.e., $\beta_\bse(x) <\tau_\bse(x)$, $\forall x\in (a,b)$.  The conditions $\mathbf{G_1},\mathbf{G_2}, \mathbf{G_3}$ imply that there exists  a compact subinterval $I=[a_*,b_*]\subset (a,b)$ such that  the  line  $L_{\xi,c}$ intersects   the graphs  of $\beta_\bse$  and $\tau_\bse$ transversally in  \emph{regular} points  on these graphs whose $x$-coordinates are contained in the interval $I$.  Denote by $\bsI_\beta^0$ (resp. $\bsI_\tau^0$) the  intersection of $L_{\xi,c}$ with the graph  of $\beta_\bse$ (resp. $\tau_\bse$.)  The superscript $0$ in $\bsI_\beta^0$ comes from our convention that $S = S_0$.   Finally we set 
 \[
 \bsI^0 := \bsI_\beta^0 \cup \bsI_\tau^0
 \]
 
 Fix a small   positive real number  $r$ with the following properties (see Figure \ref{fig: isolate}). 
 
 \begin{itemize}
 
 \item  The  closed squares  $\Sigma_r(\bp)$, $\bp\in\bsI^0$ are pairwise disjoint.
 
 \item   For each point $\bp\in \bsI^0$, there exists  $i=0,1,\dotsc, n$ such that the projection  of $\Sigma_r(\bp)$ onto the $x$-axis  is contained in a compact sub-interval $J_{\bp}\subset (c_{i-1}, c_i)$.
 
 \end{itemize}
 
 We set
 \[
 \Sigma_r(\bsI_\beta^0):= \bigcup_{\bp\in\bsI_\beta} \Sigma_r(\bp),\;\;\Sigma_r(\bsI_\tau^0):=\bigcup_{\bp\in\bsI_\tau} \Sigma_r(\bp).
 \]

\begin{lemma} (a)  Denote by $\bsI_\beta^\ep$ the intersection of $L_{\xi,c}$ with the graph of $\beta_{\bse,\ep}$.     There exists $\ep_\beta>0$  with the following properties. 

\begin{itemize}
\item[(a1)] For any $\ep\leq \ep_\beta$  we have
\[
\bsI_\beta^\ep\subset \Sigma_r(\bsI_\beta),
\]
\item[(a2)] For any  $\bp\in \bsI_\beta^0$    and any $\ep\leq \ep_\beta$  the line  $L_{\xi,c}$ intersects the portion of the graph of $\beta_{\bse,\ep}$ inside $\Sigma_r(\bp)$ in a unique point $\bp(\ep)$.  This intersection  is transversal and
\begin{equation}
\sign(\bp, \beta_{\bse,0})=\sign(\bp(\ep),\beta_{\bse,\ep}).
\label{eq: sign=}
\end{equation}
\end{itemize}

(b) Similar  statements are true with the  functions $\beta_{\bse,\ep}$ replaced by the top functions $\tau_{\bse,\ep}$.

\qed
\label{lemma: cross}
\end{lemma}

 We defer the proof of this result to the end of this section.   
 
 Set $\bsI^\ep=\bsI^\ep_\beta\cup \bsI^\ep_\tau$.  We will refer to the intersection  of $L_{\xi,c}$ with $\pa C^\ep_{\bse,0}$   as the $0$-crossing set   and, for $\ep > 0$, we we will refer to the   intersection of $L_{\xi,c}$ with $\pa C_{\bse,\ep}$ as the $\ep$-crossing set.   (The set $C^\ve_{\bse,0}$ is defined in (\ref{eq: cve0}).)   For $\ep\geq 0$ we denote by $\bsK_\ep$ the $\ep$-crossing set.
 
 Observe that the set $\bsI^\ep$  is contained in $\bsK_\ep$  but the  $\ep$-crossing set may contain additional points, namely, the intersection of $L_{\xi,c}$ with the vertical lines $x=a_\ep,b_\ep$.  The Hausdorff distance between the $\bsK_\ep$ and $\bsK_0$   goes to  zero as $\ep\searrow 0$.  Moreover, Lemma \ref{lemma: cross} implies that for any $\ep$ sufficiently small there exists a bijection
 \[
 \Psi_\ep: \bsK_0\ra \bsK_\ep
 \]
 defined by 
 \[
 \Psi_\ep(\bp) =\bsK_\ep\cap\Sigma_r(\bp).
 \]
 For $\ep>0$ we denote by $C_{\ep}^+$ the intersection of $C_{\bse,\ep}$ with the half-plane $\{\xi\geq c\}$. Similarly, we define $C_0^+$ to be the intersection of $C^\ep_{\bse,0}$ with the same half-plane.   We have to prove that
 \begin{equation}
 \chi(C_\ep^+)=\chi(C_0^+),\;\;\forall \ep\ll 1.
 \label{eq: final-conv3}
 \end{equation}
 For $\ep\geq 0$ the connected components   of $C_{\ep}^+$ are all homeomorphic to closed $2$-dimensional disks so that  the Euler characteristic  of $C_{\bse,\ep}^+$ is equal to the number of boundary components  of  $\pa C^+_\ep$.
 
 Observe that if $\bsK_0=\emptyset$, then $\bsK_\ep=\emptyset$ for all $\ep$ sufficiently small. In this case  $C_{\bse,\ep}^+$  is homeomorphic to a closed disk for all $\ep$ sufficiently small and  (\ref{eq: final-conv3}) is obviously true. We need to investigate the case $\bsK_0\neq \emptyset$.
 
 For $\ep> 0$ we define an equivalence relation $\sim_\ep$  on $\bsK_\ep$ by declaring $\bp\sim_\ep \bq$ if and only if $\bp$ and $\bq$ belong   to the same component of $\pa C_\ep^+$. Similarly, we define an equivalence relation $\sim_0$ on $\bsK_0$  by declaring   $\bp\sim_0\bq$ if and only $\bp$ and $\bq$ belong to the same connected component of    $\pa C_0^+$.  Thus, for $\ep\geq 0$ the number of connected components of $\pa C_\ep^+$ is  equal to the number of equivalence classes of $\sim_\ep$. To prove the equality (\ref{eq: final-conv3}) it suffices to show that 
 \begin{equation}
 \bp\sim_0\bq \Rightarrow \Psi_\ep(\bp)\sim_\ep\Psi_\ep(\bq).
 \label{eq: equiv}
 \end{equation}
 Indeed, if (\ref{eq: equiv}) holds, then  we deduce  that the number of equivalence classes  of $\sim_\ep$ is     not larger than the number of equivalence classes of $\sim_0$.   Since the number of connected components of $C_0^+$ is independent of $\ep$ if $\ep$ is small   and
 \[
 \dist(C_\ep^+,C_0^+)\ra 0\;\;\mbox{as}\;\;\ep\searrow 0
 \]
 we deduce that $C_{\bse,\ep}^+$ has at least as many components as $C_0^+$.
 
 Fix a component $R$ of $C_0^+$ and $\bp,\bq\in \pa R$.    We denote by  $[\bp,\bq]_R$ the   arc  of $\pa R$ obtained by traveling counterclockwise   from $\bp$ to $\bq$. Along this arc there could be additional crossing points $\bp_0=\bp, \dotsc, \bp_k=\bq\in \bsK_0$, arranged in counterclockwise order.  We set
 \[
 \bp_j^\ep:=\Psi_\ep(\bp_j).
 \]
 Each of the arcs $[\bp_{j-1},\bp_j]_R$ is of one of the following  two types:
 
 \begin{itemize}
 \item[I.]  A  line segment contained in  $L_{\xi,c}$.
  
 \item[II.]   A sub-arc   of  $\pa C^\ep_{\bse,0}$  that intersects $L_{\xi,c}$ only at its endpoints. 
 \end{itemize}
 
 \begin{figure}[ht]
\centering{\includegraphics[height=2in,width=3in]{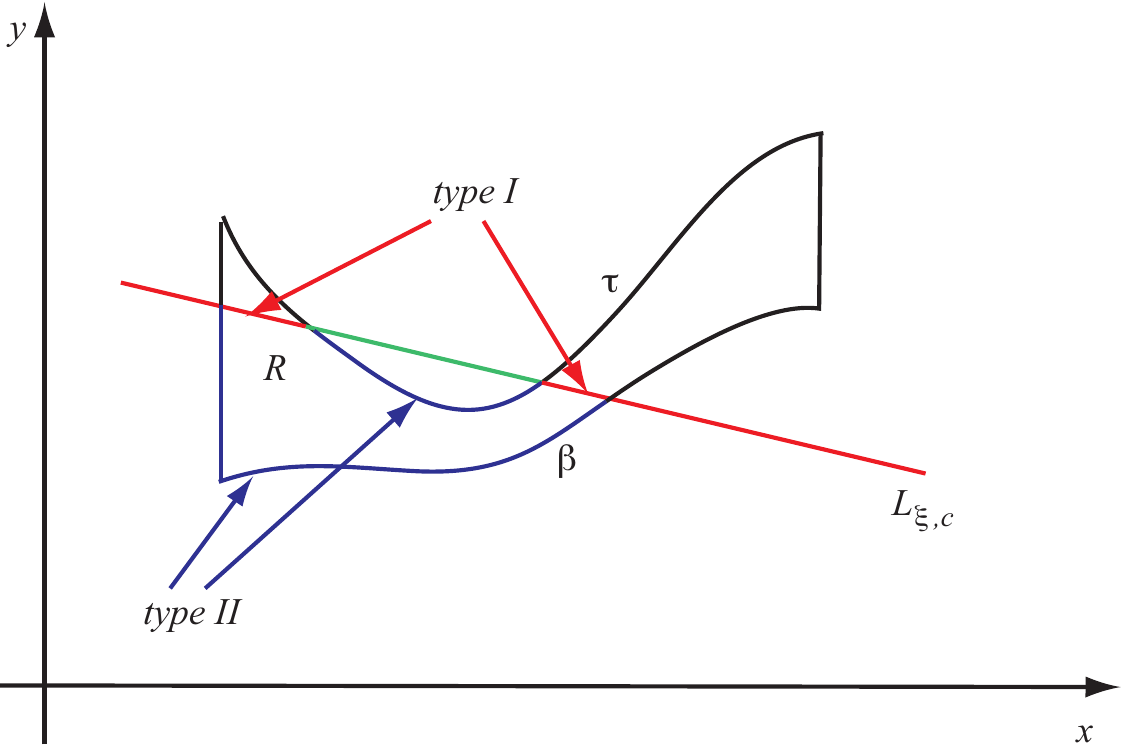}}
\caption{\sl An elementary region $S(\beta,\tau)$ cut  by a line $L_{\xi,c}$. The intersection of this region with the  lower half-plane  determined by $L_{\xi,c}$ has one component $R$ whose  boundary is decomposed in arcs of two types.}
\label{fig: elem-slice}
\end{figure}

 If $[\bp_{j-1}, \bp_j]_R$ is of type  I,  so that it is contained in $L_{\xi,c}$,   then  the points $\bp_{j-1}^\ep$ and $\bp_j^\ep$ are also contained in $L_{\xi,c}$  and we denote by $[\bp_{j-1}^\ep, \bp_j^\ep]_R$ the oriented line segment going from  $\bp_{j-1}^\ep$ to $\bp_j^\ep$. Clearly 
 \[
 p_{j-1}^\ep\sim_\ep\bp_j^\ep.
 \]
Suppose now that $[\bp_{j-1}, \bp_j]_R$ is of type  II.   The points $\bp_{j-1}^\ep$ and $\bp_j^\ep$ divide  the  boundary $\pa C_{\bse,\ep}$ into two arcs, one of which  approaches $[\bp_{j-1}, \bp_j]_R$ in the Hausdorff distance as $\ep\ra 0$. We denote this arc by $[\bp_{j-1}^\ep, \bp_j^\ep]_R$. Lemma \ref{lemma: cross} implies that the arc $[\bp_{j-1}^\ep,\bp_j^\ep]_R$ intersects $L_{\xi,c}$ only at its end points if $\ep$ is sufficiently small. For such $\ep$'s  the arc $[\bp_{j-1}^\ep,\bp_j^\ep]_R$ lies on the same side of $L_{\xi,c}$  as $[\bp_{j-1}, \bp_j]_R$ so that  $p_{ji1}^\ep\sim\bp_j^\ep$.    By transitivity we now deduce that 
\[
\Psi_\ep(\bp)=\bp_0^\ep\sim_\ep \bp_k^\ep=\Psi_\ep(\bq).
\]
This proves (\ref{eq: equiv}) and  thus proves (\ref{eq: final-conv2}) in the case when  the elementary set $C_{\bse,0}$ is nondegenerate.

\medskip

\noindent {\bf Case 2.} \emph{The elementary set $C_{\bse,0}$ is degenerate}, i.e., $\beta_{\bse,0}=\tau_{\bse,0}$. We denote by $\bsJ^0$ the set consisting of the  endpoints of the graph of $\beta_{\bse, 0}$ and the intersection of this graph with $L_{\xi,c}$. Similarly, denote by $\bsJ^\ep_\beta$ (resp. $\bsJ^\ep_\tau$) the set consisting of the endpoints of the graph of $\beta_{\bse,\ep}$ (resp. $\tau_{\bse,\ep}$)  and the intersection of this graph  with the line $L_{\xi,c}$.  As in  {\bf Case 1} we  can invoke Lemma \ref{lemma: cross} to obtain bijections
\[
\Psi_\ep^\beta:\bsJ^0\ra \bsJ^\ep_{\beta},\;\;\Psi_\ep^\tau:\bsJ^0_+\ra \bsJ^\ep_{\tau}.
\]
We continue to use the notations $C_0^+$ and $C_{\ep}^+$  introduced in the proof of {\bf Case 1}.    In this case $C_0^+$ is a finite union of subarcs  of the graph of $\beta_{\bse,0}$.      Let these arcs  be $A_1,\dotsc, A_k$.   Each of these arcs  carry a natural orientation.  Denote by $\bp_j$ the initial  point of $A_j$ and by $\bq_j$ the final point   of $A_j$. We set
\[
\bp_j^\beta(\ep):= \Psi^\beta_\ep(\bp_j),\;\; \bp_j^\tau(\ep):= \Psi^\tau_\ep(\bp_j).
\]
We define $\bq_j^\beta(\ep)$ and $\bq_j^\tau(\ep)$ in a similar  fashion. Consider the simple closed curve $A_j^\ep$  which is the union of  four arcs (see Figure \ref{fig: elem-slice1}).

\begin{itemize}

\item The line segment from $\bp_j^\tau(\ep)$ to $\bp^\beta_j(\ep)$.

\item The arc of $\beta_{\bse,\ep}$ from $\bp^\beta_j(\ep)$ to $\bq_j^\beta(\ep)$.

\item The line segment from   $\bq_j^\beta(\ep)$ to  $\bq_j^\tau(\ep)$.

\item The arc of $\tau_{\bse,\ep}$  from  $\bq_j^\tau(\ep)$ to  $\bp_j^\tau(\ep)$.

\end{itemize}

\begin{figure}[ht]
\centering{\includegraphics[height=1.5in,width=3.6in]{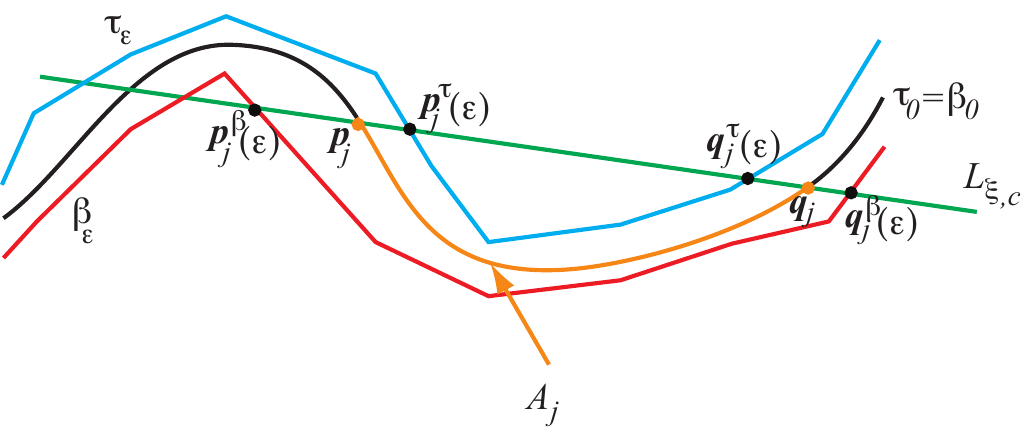}}
\caption{\sl The arc $A_j$ and the simple closed curve $A_j^\ep$, $\bp_j^\tau(\ep)\ra\bp_j^\beta(\ep)\ra \bq_j^\beta(\ep)\ra \bq_j^\tau(\ep)\ra \bp_j^\tau(\ep)$.}
\label{fig: elem-slice1}
\end{figure}

Lemma \ref{lemma: cross} implies that for $\ep$ sufficiently small the  closed curve $A_j^\ep$ is contained entirely in the half-plane $\{\xi\geq c\}$ so the   bounded region it surrounds is contained in this half-plane as well.  The  region  $C_{\ep}^+$ consists precisely of the regions surrounded by  the closed curves $A_j^\ep$, $j=1,\dotsc, k$ so that $\chi(C_0^+)=\chi(C_\ep^+)=k$ for all $\ep$ sufficiently small. This  proves (\ref{eq: final-conv2}) in  {\bf Case 2}.

\medskip

\noindent {\bf Case 3.} \emph{$C_{\bse,0}$ is a mixed elementary set}.   It has a \emph{minimal} good partition
\[
a=t_0<t_1<\cdots <t_n=b,\;\;n\geq 2,
\]
where for   each $j=1,\dotsc, n$ the intersection  of $C_{\bse,0}$ with the  strip $[t_{j-1},t_j]\times \bR$ is  either degenerate or nondegenerate. The intersection  of the graphs of $\beta$ and $\tau$ with each of the vertical lines $x=t_j$, $j=0,\dotsc ,n$, is  a singular point of $C_{\bse, 0}$; see Figure \ref{fig: eul-part}. Since the $c$ is not a critical value of the restriction of $\xi$ to $C_{\bse,0}$, we conclude that the line $L_{\xi, c}$ does not contain any of these  singular points.

  \begin{figure}[ht]
\centering{\includegraphics[height=1.3in,width=3.8in]{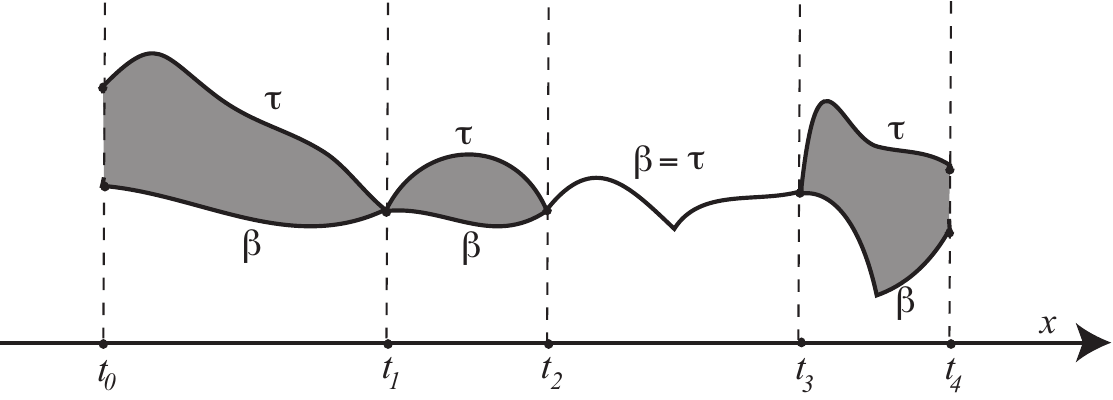}}
\caption{\sl The minimal good partition of a mixed elementary sets.}
\label{fig: eul-part}
\end{figure}

Set $H^+_{\xi,c}:=\{\xi\geq c\}$. For $j=1,\dotsc, n$ and $\ep>0$   we set
\[
R_i:= \bigl([t_{j-1},t_j]\times \bR\bigr)\cap C^\ep_{\be, 0},\;\;R_{i,\ep}:= \bigl([t_{j-1},t_j]\times \bR\bigr)\cap C_{\be, \ep}.
\]
For $k=1,\dotsc, n-1$  and $\ep>0$ we set
\[
V_k:= \{x=t_k\}\cap C^\ep_{\be, 0},\;\; V_{k,\ep}:=  \{x=t_k\}\cap C_{\be, \ep}.
\]
To prove  (\ref{eq: final-conv2}) it suffices to show that
\begin{subequations}
\begin{equation}
\chi(R_j\cap H^+_{\xi,c})=\chi(R_{j,\ep}\cap H^+_{\xi,c}),\;\;\forall \ep\ll1,\;\;j=1,\dotsc, n,
\label{eq: final-conv3a}
\end{equation}
\begin{equation}
\chi(V_k\cap H^+_{\xi,c})=\chi(V_{k,\ep}\cap H^+_{\xi,c}),\;\;\forall \ep\ll1,\;\;k=1,\dotsc, n-1,
\label{eq: final-conv3b}
\end{equation}
\end{subequations}

The equalities (\ref{eq: final-conv3a}) follow from the  Cases 1 and 2 investigated above.  The equalities   (\ref{eq: final-conv3b})    are consequences of  the following   simple facts.

\begin{itemize}

\item  For any $k=1,\dotsc, n-1$,  the set   $V_k$ consists of a single point  that does not line on the line $L_{\xi,c}$.

\item  For any $k=1,\dotsc, n-1$, and $\ep>0$  the set   $V_{k,\ep}$ consists of a single  vertical line segment.

\item  For any $k=1,\dotsc, n-1$, the set $V_{k,\ep}$ converges in the Hausdorff metric to the set $V_k$. In particular, for $\ep\ll1$ we have
\[
V_k\subset  H^+_{\xi,c} \Llra  V_{k,\ep}\subset  H^+_{\xi,c}.
\]
\end{itemize}

This  completes the proof of  Theorem \ref{th: main}.

\end{proof}

 \noindent {\bf Proof of Lemma \ref{lemma: cross}.} The inclusion (a1) follows from the fact that the  distance between the graph of $\beta_{\bse,\ep}$ and the graph of $\beta_{\bse,0}$ approaches zero as $\ep\ra 0$. To prove  (a2) let us denote by $(x_0,y_0)$ the coordinates of $\bp$.  From the choice of $r$ we deduce that  for  $\ep$ sufficiently small the interval
 \[
 J_\ep:= [x_0-r-\si(\ep)\ep, x_0+r+\si(\ep)\ep]
 \]
 is contained entirely in an interval of the form $(c_{j-1},c_j)$  for some $j=1,\dotsc, n$ (where $c_j$ were defined much earlier in the proof to be the $x$-coordinates such that either $\beta_\ep$ or $\tau_\ep$ fail to be real analytic) so that $\beta_{\bse,0}$ is $C^2$ on  $J_\ep$.  We set
 \[
 K_1=\sup_{x\in J_\ep} |\beta'_{\bse,0}(x)|,\;\;K_2=\sup_{x\in J_\ep} |\beta''_{\bse,0}(x)|.
 \]
We denote by $m_\xi$ the slope of $L_{\xi,c}$ and by $m_0$ the slope of the tangent to the graph of $\beta_{\bse,0}$ at $\bp$,
\[
m_0=\beta'_{\bse,0}(x_0).
\]
Because $L_{\xi,c}$ intersects the graph of $\beta_{\bse,0}$ transversally at $\bp$ we deduce $m_0\neq m_\xi$. We deduce that for every $x\in J_\ep$ we have
\begin{equation}
|\beta'_{\bse,0}(x)-m_0|\leq K_2|x-x_0|.
\label{eq: var1}
\end{equation}
The    function $\beta_{\bse,\ep}$ is piecewise linear.    Consider  the portion of this graph
\[
\lan \bp_0^\ep,\dotsc, \bp_{\ell(\ep)}^\ep\ran
\]
 with the property that  $\bp_1^\ep,\dotsc, \bp^\ep_{\ell(\ep)-1}$ are successive  vertices on the graph of $\beta_{\bse,\ep}$ such that
 \[
 \lan \bp_1^\ep,\dotsc, \bp_{\ell(\ep)-1}^\ep\ran \subset \Sigma_r(\bp),\;\;\bp_0^\ep,\bp^\ep_{\ell(\ep)}\not\in \Sigma_r(\bp).
 \]
 Denote by $(x_j^\ep,y_j^\ep)$ the coordinates of $\bp^\ep_j$, $j=0,\dotsc,\ell(\ep)$, and set $z_j^\ep=\beta_{\bse,0}(x_j^\ep)$.
 Observe that
 \[
 |y_j^\ep-z_j^\ep|\leq (K_1+4)\ep,\;\;\forall j=0,\dotsc, \ell(\ep),\;\; \frac{1}{x_k^\ve-x^\ve_{k-1}}=O\left(\frac{1}{\ep\si(\ep)}\right),\;\;\forall 1\leq k\leq \ell(\ep).
 \]
 We deduce that
 \[
m_j^\ep:= \frac{y^\ep_j-y^\ep_{j-1}}{x_j^\ep-x_{j-1}^\ep}=  \frac{z^\ep_j-z^\ep_{j-1}}{x_j^\ep-x_{j-1}^\ep} +O\left(\frac{1}{\si(\ep)}\right),
 \]
 where the constant implied by the $O$-symbol is independent of $\ep$.  The mean value theorem implies that  the difference quotient in the right-hand side of the above equality is equal to the derivative of $\beta_{\bse,0}$ at a point $\eta_j^\ep\in (x_{j-1}^\ep,x_j^\ep)$. Using (\ref{eq: var1}) we deduce that
 \begin{equation}
 \left|m_j^\ep-m_0\right|= O\left( \frac{1}{\si(\ep)} +|x^\ep_{j-1}-x_0|+|x_j^\ep-x_0|\right).
 \label{eq: var2}
 \end{equation}
 Since $\si(\ep)\ra \infty$ we deduce that  given 
 \[
 \gamma <\min\left\{ r, \frac{1}{4}|m_\xi-m_0|\,\right\}
 \]
  there exist  constants  $\delta=\delta(\gamma)>0$  and $\ep(\gamma)>0$ such that, for any $\ep<\ep(\gamma)$ the segments of the graph of  $\beta_{\bse,\ep}$ situated in the strip $|x-x_0|\leq \gamma$ have slopes  $m_j^\ep$ located in the range $(m_0-\gamma, m_0+\gamma)$. In particular, none of these  slopes   can be equal to $m_\xi$, and they are all situated on the same  side of $m_\xi$ as $m_0$.

If we fix $\gamma$ as above we can find  $\ep_1(\gamma)>0$ such that, for $\ep<\ep_1(\gamma)$ all the intersection points of $L_{\xi,c}$ with the graph of $\beta_{\bse,\ep}$ located in $\Sigma_r(\bp)$ are in fact located in the narrow vertical strip $|x-x_0|<\gamma$.  The above discussion then shows that all  these intersections must be transversal and  they all have the same sign, $\sign(\bp, \beta_{\bse,0})$.  Denote by $N_\ep(\gamma)$ the number of such intersections.  Set
\[
P_0^\pm= \bigl(\,x_0\pm\gamma, \beta_{\bse,0}(x_0\pm\gamma)\,\bigr),\;\;P_\ep^\pm = \bigl(\,x_0\pm \gamma, \beta_{\bse,\ep}(x_0\pm\gamma)\,\bigr).
\]
Consider  now the closed curve   $C^\ep_\gamma$ obtained  as follows.

\begin{itemize}
\item Travel from the point $P_\ep^-$ to  $P_\ep^+$ along the graph  of $\beta_{\bse,\ep}$ .

\item Next, travel   on  the vertical segment connecting  $P_\ep^+$ to $P_0^+$.

\item Travel along the graph of $\beta_{\bse,0}$ from $P_0^+$ to $P_0^-$.

\item  Finally, travel  along the vertical segment connecting $P_0^-$ to $P_\ep^-$.

\end{itemize}

The above  discussion shows that the intersection number between the line $L_{\xi,c}$ and the curve $C_\gamma^\ep$ is $\pm \bigl(\,N_\ep(\gamma)-1\,\bigr)$. On the other hand, since this curve is   homologically trivial we deduce  that the intersection number $L_{\xi,c}$ and $C_\gamma^\ep$ is $0$.  Therefore we conclude that $N_\ep(\gamma) = 1$, which completes the lemma in the case of the function $\beta_{\bse,\ep}$.  

The above proof can be repeated replacing $\beta_{\bse,\ep}$ with $\tau_{\bse,\ep}$ for the upper boundary case. \qed

\appendix

\section{Semialgebraic geometry}
\label{s: a}
\setcounter{equation}{0}

A set   $X\subset \bR^n$ is  called \emph{semialgebraic}  if  it can be written as a finite union
\[
X= X_1\cup \cdots \cup X_N,
\]
where each of the sets $X_i$ is described by a  finite system of polynomial   inequalities.     

A map $F:X_0\ra X_1$ between two semialgebraic sets $X_i\in \bR^{n_i}$, $i=0,1$, is called \emph{semialgebraic}  if  its graph  $\Gamma_F$ is a semialgebraic   subset of $\bR^{n_0+n_1}$.

Here is a list of basic properties of  semialgebraic sets and functions. For proofs and more details we refer to \cite{BCR, Co, Dr}.

\begin{itemize}

\item The union, the intersection   and the Cartesian product of two semialgebraic sets  are semialgebraic.

\item  If $X,Y$ are semialgebraic subsets     of $\bR^n$ then so is their difference.

\item     A subset of $\bR$ is semialgebraic if and only if  it is a finite union of open interval and points.

\item (\emph{Tarski-Seidenberg})   The image and preimage of a semialgebraic set via a semialgebraic  map are semialgebraic sets.

\item  If  $I$ is an interval of the  real axis   and $f: I\ra \bR$ is semialgebraic, then there exists a finite subset $F\subset I$ such that the restriction of $F$ to any component of $I\setminus F$ is monotone and real analytic.

\item (\emph{Curve selection}) If $X$ is a  semialgebraic subset of $\bR^n$ and $x_0\in \cl(X)\setminus X$, then there exists a continuous semialgebraic map $\gamma: (0,1)\ra X$ such that
\[
\lim_{t\searrow 0} \gamma(t)=x_0.
\]

\item (\emph{{\L}ojasewicz' inequality})  Suppose that $X$ is a compact semialgebraic  set and $f, g:X\ra \bR$ are continuous semialgebraic   functions such that
\[
\{f=0\}\subset \{ g=0\}.
\]
Then there exists a positive integer $N$ and a positive real number $C$ such that
\begin{equation}
|g(x)|^N\leq C|f(x)|,\;\;\forall x\in X.
\label{Loja}
\end{equation}

\item Suppose that $X$ is a compact  semialgebraic  set and $f: X\ra \bR$ is a continuous semialgebraic function.   Then the function $\bR\ra \bZ$ that associates to each $t\in \bR$ the Euler characteristic  of the level set $\{f=t\}$ is a semialgebraic function.

\item A semialgebraic set is connected if and only if  it is path connected.

\item  A semialgebraic set has finitely many  connected components and each of them is also a semialgebraic set.

\end{itemize}

\noindent {\bf Proof of Proposition \ref{prop: pl-approx}}  We prove only the statement about  the total curvature. The statement  about the   perimeter   follows the same pattern  and has fewer complications.  First some terminology.

A continuous function  $f: [a,b]\ra \bR$ is said to be piecewise $C^2$ if  there exists a finite set
\[
S=\bigl\{a=s_0<s_1<\cdots < s_\ell=b\,\bigr\},
\]
such that  for any $j=1,\dotsc, \ell$, and any $k=1,2$  the restriction of $f$ to the open interval $(s_{j-1},s_j)$ is a $C^2$ function and the limits
\[
\lim_{x\searrow s_{j-1} }f'(x),\;\;\lim_{x\nearrow s_j} f'(x)
\]
exist. Note that the  last condition implies that  as $s\to s_j\pm 0$ the oriented  tangent space  to the graph of $f$ at $(s,f(s))$ has a limit  in the Grassmanian of oriented one-dimensional subspaces of $\bR^2$.

We  say that  the arc   $\bsC$ is \emph{convenient}   if there exists a piecewise $C^2$-function $f:[a,b]\ra \bR$  and  an orthonormal  system of coordinates $(\bar{x},\bar{y})$ on $\bR^2$ such that   
\[
\bsC=\bigl\{\, (\bar{x},\bar{y});\;\;\bar{y}=f(\bar{x}),\;\;\bar{x}\in [a,b]\,\bigr\}.
\]
  When  $\bsC$ is  convenient, Proposition \ref{prop: pl-approx}  is a special case of  \cite[Prop. 3.6]{Row}.

To deal with the general case let us observe that since $\bsC$ is  \emph{semialgebraic  arc}, for any point $\bp\in C$ there exists a closed disk $D$ centered at $\bp$ such that the intersection $D\cap C$  is  a convenient arc. (As $\bar{x}$-axis we can choose any line that is not perpendicular to the lines in the tangent cone to $\bsC$ at $\bp$ described in Definition \ref{def: tcone}. ) 

Since any continuous semialgebraic function is piecewise $C^2$ we  deduce  there exists an ordered  sampling  of $\bsC$
\[
\eQ=\bigl\{ Q_0,\dotsc, Q_N\,\bigr\}
\]
with the following properties.

\begin{itemize}

\item[(a)] The arc $\bsC$ starts at $Q_0$ and ends at $Q_N$.

\item[(b)]   The arc $\bsC$   is smooth at each of the points $Q_1,\dotsc, Q_{N-1}$.

\item[(c)] For any $j=1,\dotsc, N$, the portion of $\bsC$ between $Q_{j-1}$ and $Q_j$  is convenient. We denote by  $\bsC_j$ this portion.

\end{itemize}

Denote by $\eP_\ep^j$ the ordered sampling of $\bsC_j$  determined by the points in $\eP_\ep$ contained in  $\bsC_j$. We denote by $K_\ep^j$ the total curvature of the $PL$-curve $\bsC(\eP_\ep^j)$.      Since each of the curves $\bsC_j$ is convenient we have
\[
\lim_{\ep\searrow 0} K_\ep^j= K(\bsC_j),\;\;\forall j=1,\dotsc, N,
\]
so that
\[
\lim_{\ep\searrow 0} \sum_{j=1}^N K_\ep^j=K(\bsC).
\]
On the other hand, since  $\bsC$ is smooth at the points $Q_1,\dotsc, Q_{N-1}$ we deduce that 
\[
\lim_{\ep\searrow 0}\left( K_\ep- \sum_{j=1}^N K_\ep^j\right)=0.
\]
\qed

\section{The approximation algorithm}
\label{s: b}
\setcounter{equation}{0}

In this section we  give a more formal description of the approximation algorithm.  

Assume that $S$ sits on a screen consisting of $m\times m$ pixels so that $\ve=\frac{1}{m}$. We convert $P_\ep(S)$ into an $m \times m$ matrix $A$  of $1$'s and $0$'s, where $A[i,j]=1$ if and only if the pixel of center $c_{i,j}(\ep)$ touches $S$. 

Given this matrix $A$ we will generate a $PL$ set $S_\ep$ which approximates the original set $S$.  We will  assume that $\ep$ is fixed throughout the description of the algorithm.

The algorithm     depends on two parameters, both positive integers:  the spread   $\sigma$ and the noise width $\nu$ which we regard as functions of $m$.  These should 
 be chosen  so that 
\[
\lim_{m \to \infty } \frac{\sigma(m)}{m} = 0, \;\;\lim_{m\to\infty} \frac{(\sigma(m))^2}{m} = \infty,
\]
\[
\lim_{m\to \infty} \frac{\nu(m)}{m} = 0, \;\;\lim_{m\to\infty} \frac{\nu(m)}{ m^{1-\kappa_0}} = \infty,
\]
where $\kappa_0\in (0,1]$ is a constant dependent on $S$ introduced in Theorem \ref{thm: noisebound}.  However, for most applications  we can choose $\kappa_0=\frac{1}{2}$ and then
\[
\sigma\approx m^{\frac{1}{2}+ s},\;\;\nu\approx m^{\frac{1}{2}+r},\;\;s,r\in \left(0,\frac{1}{2}\right).
\]

The algorithm  uses several smaller subroutines.  The first subroutine $\stack$ obtains information about the various columns of $A$ which will be used to determine both the noise intervals as well as to select the vertices of $S_\ep$.  The input of $\stack$ is a list
\[
C = c_1, \dotsc c_m,\;\; c_i = 0,1,
\]
where $C$ is one of the columns of $A$.  The output  of $\stack$  is a  list  of  nonnegative integers
\[
\bn(C); \;\;b_1\leq t_1<b_2\leq t_2<\cdots <b_{\bn(C)}\leq t_{\bn(C)},
\]
where  $\bn(C)$ is the number of stacks in the column encoded by    $C$,   and the location of the bottom and top pixel in the $j$-th stack is determined by    the integers  $b_j, t_j$. More formally
\[
c_k=1 \Llra \exists 1\leq j\leq \bn(C):\;\; b_j\leq k\leq t_j.
\]
  If  $C=C_i$, the $i$-th column  of $A$, i.e.,
  \[
  C_i= a_{i,1},\dotsc, a_{i,m},
  \]
 then we will denote the output $\stack(C_i)$ by
 \[
 \bn_i,\;\; b_{i,1}\leq t_{i,1}<\cdots <b_{i,\bn_i}\leq t_{i,\bn_i}.
 \]
 A number $1\leq i\leq m-1$ is called a \emph{jump point} if
 \[
 \bn_i\neq \bn_{i+1}.
 \]
 
  \medskip
  
The next subroutine is called $\jump$.  Its input is an integer $k\in [1, m)$  and the  output is an integer $j_k=\jump(k)$ where $j_k$ is the next jump point, i.e., if
\[
\bigl\{ i\in [k,m)\cap\bZ;\;\; i\;\mbox{is a jump point}\,\bigr\}=\emptyset,
\]
 then we set
\[
\jump(k):= m+1.
\]
Otherwise
 \[
 \jump(k)=\min\bigl\{ i\in [k,m)\cap\bZ;\;\; i\;\mbox{is a jump point}\,\bigr\}.
 \]

\medskip 
 
Using these subroutines we can construct the \emph{noise regions} of the approximation.  These are simply the columns which are within $2\nu$ columns of a jump point.  Specifically we create a certain number of intervals:
\[
[\ell_1, r_1],\dotsc, [\ell_\alpha, r_\alpha]\subset [1,m]
\]
where the integers  $\ell_k, r_k$ are determined inductively as follows.
\[
\ell_1= \max\bigl(\, \jump(1)-2\nu, 1\,\bigr),\;\; r_1=\min\bigl(\, m,  \jump(1)+2\nu\,\bigr).
\]
Suppose that $\ell_1,r_1,\dotsc, \ell_j,r_j$ are determined. If $\jump(r_j)>m$    we stop. Otherwise  we set
\[
\ell_{j+1}= \max\bigl(\, \jump(r_j)-2\nu, 1\,\bigr),\;\;r_{j+1}=\min\bigl(\, m,  \jump(r_j)+2\nu\,\bigr).
\]
The intervals $[\ell_1,r_1],\dotsc ,[\ell_\alpha,r_\alpha]$ may not be disjoint,  but their union is a   \emph{disjoint} union of intervals
\[
[a_1, b_1],\dotsc, [a_J,b_J],\;\; b_i<a_{i+1}.
\]
The intervals $[a_j,b_j], 1\leq j\leq J$ are the  \emph{noise intervals}.  The intervals
\[
[1,a_1], [b_1,a_2],\dotsc, [b_{J-1},a_J], [b_J, m]
\]
are  the \emph{regular intervals}.

\medskip

Now that we have determined the noise and regular intervals, we can create the approximation $S_\ep$.  We do this with separate procedures on the noise or regular intervals.  In either case the approximation will be formed by (possibly degenerate) trapezoids whose bases are vertical. We call any set which is a union of finitely many such trapezoids a \emph{polytrapezoid}.  The approximations on the regular and noise intervals will both be polytrapezoids, and $S_\ep$ itself will also be a polytrapezoid.

First some notation. Given   a collection of  points
\[
B_0, T_0,\dotsc, B_N, T_N\in \bR^2
\]
such that
\[
x(B_i)= x(T_i),\;\;y(B_i)\leq y(T_i),\;\;\forall i=0,\dotsc, N,
\]
\[
x(B_{j-1})<x(B_{j}),\;\;\forall 1\leq j\leq N,
\]
  we denote by $\polygon(B_0, T_0, \dotsc, B_N, T_N)$ the   region surrounded by the simple closed $PL$-curve obtained    as the union of line segments
  \[
  [B_0, B_1],\dotsc, [B_{N-1}, B_N], 
  \]
  \[
  [B_N,T_N],\dotsc, [T_1,T_0], [T_0, B_0].
  \]
 Note that  each of the quadrilaterals $B_{i-1}B_iT_iT_{i-1}$ is a (possibly degenerate) trapezoid with vertical bases.
 
 \medskip
 
Consider first the regular intervals.   Given a regular interval $I:=[p,q]$  we observe that  the number of stacks $\bn_i$ is independent of $i\in [p,q]$. We denote this shared number by $\bn=\bn(I)$.

 We  construct inductively a sequence of numbers $i_0 <\cdots < i_N$  as follows: 
 \begin{itemize}
 
 \item We set $i_0=p$. 
 
 \item If $q-p<2\sigma$ we set $N=1$ and $i_1=q$. 
 
 \item If  $i_0,\dotsc, i_k$ are  already constructed, then,  if $q-i_k< 2\sigma$ we set $N=k+1$ and $i_{k+1}=q$,  else $i_{k+1}=i_k+\sigma$.
 \end{itemize}
 
 Note that if $q-p> \sigma$, then $N\geq 1$, $i_0=p$, $i_N=q$ and
 \[
 N=1\;\;\mbox{if}\;\; q-p<\sigma.
 \]
 We have
 \[
 \stack(C_{i_k}) = \bn, \;\;b_{i_k,1}, t_{i_k,1},\dotsc, b_{i_k, \bn}, t_{i_k, \bn}.
 \]
 For $j=1,\dotsc, \bn$, and $k=0,\dotsc, N$ we denote by $B_{k,j}$ the center  of the $\ep$-pixel corresponding to the element entry $b_{i_k,j}$ in the column $C_{i_k}$. Similarly we denote by $T_{k,j}$ the center of the pixel corresponding to the entry $t_{i_k,j}$ of the column $C_{i_k}$.  For $1\leq j\leq \bn(I),
$ we set
 \[
 \eP_j(I) := \polygon(B_{0,j}, T_{0,j}, \dotsc,   B_{N,j}, T_{N,j}).
  \]
  Define
 \[
 \eP(I)=\bigcup_{j=1}^{\bn(I)} \eP_j(I),\;\; \eP_{\mathrm{reg}}:=\bigcup_{I\;\mathrm{regular\; interval}} \eP(I).
 \]
 
\medskip

\noindent Suppose now that $I=[p,q]$ is a noise interval.  We  modify the column 
 \[
 C_p= a_{p,1},\dotsc, a_{p,m}
 \]
  to a column 
  \[
  C'_p= a'_{p,1},\dotsc, a'_{p,m},
  \]
  by setting
  \[
  a'_{p,k}:= \begin{cases}
  1, & \mbox{if}\;\sum_{i=p}^q a_{i,k}>0\\
  &\\
  0, & \mbox{if}\;\sum_{i=p}^q a_{i,k}=0.
  \end{cases}
  \]
  We apply the subroutine $\stack$ to the new column $C'_p$and the output is
  \[
  \stack(C'_p)= \bn(I),\;\;b_1\leq t_1<\cdots <b_{\bn(I)}\leq t_{\bn(I)}.
  \]
 For $j=1,\dotsc, \bn(I)$ we  set
 \[
 B_{0,j}:=A[p,b_j],\;\; T_{0,j}:= A[p,t_j],
 \]
 \[
 B_{1,j}:=A[q,b_j],\;\; T_{0,j}:= A[q,t_j],
 \]
 where  we recall that $A[i,j]$ is defined as the center of the pixel associated to $a_{i,j}$. Next,  for $j=1,\dotsc, \bn(I)$ we  define the rectangle 
 \[
 \eR_j(I) := \polygon(B_{0,j}, T_{0,j}, B_{1,j}, T_{1,j}),
 \]
 and we set
 \[
 \eR(I)= \bigcup_{j=1}^{\bn(I)} \eR_j(I),\;\;  \eP_{\rm noise}:=\bigcup_{I\;\mathrm{noise\; interval}}\eR(I).
 \]
 The output of the algorithm is the polytrapezoid
  \[
  \eP_\ep(A):=\eP_{\rm regular}\cup \eP_{\rm noise}.
  \]


\begin{thebibliography}{XXXXXX}

\bibitem{Ber} A. Bernig: {\sl The normal cycle of compact definable sets}, Israel J. Math., {\bf 159}(2007), 373-411.

\bibitem{BCR} J. Bochnak, M. Coste, M.-F.. Roy: {\sl  Real Algebraic Geometry},   Translated from the 1987 French original. Revised by the authors. Ergebnisse der Mathematik und ihrer Grenzgebiete,   vol. 36, Springer-Verlag, Berlin, 1998.

\bibitem{CCLT} F. Chazal, D. Cohen-Steiner, A. Lieutier, B. Thibert: {\sl Stability of curvature measures}, Computer Graphics Forum,  {\bf 28}(2009), 1485-1496, \href{http://front.math.ucdavis.edu/0812.1390}{\textsf{arXiv: 0812.1390}}


\bibitem{Co} M. Coste: {\sl An Introduction to  Semialgebraic Geometry},  Dip. Mat. Univ. Pisa, Dottorato di Ricerca in Matematica, Istituti Editoriali e Poligrafici Internazionali, Pisa (2000).

\url{http://perso.univ-rennes1.fr/michel.coste/polyens/SAG.pdf}


\bibitem{EH} H. Edelsbrunner, J. Harer: {\sl Computational Topology. An Introduction},  Amer. Math. Soc., 2010.


\bibitem{GM} M. Gorseky, R. MacPherson: {\sl Stratified Morse Theory}, Springer Verlag, 1988.

\bibitem{Fu1} J. Fu: {\sl  Convergence of curvatures in secant approximations}, J. Diff. Geom. {\bf 37}(1993), 177-190.

\bibitem{Fu2} J. Fu: {\sl  Curvature measures of subanalytic sets}, Am. J. Math. {\bf 116}(1994), 819-890.

\bibitem{KR} D.A. Klain, G.-C. Rota: {\sl Introduction to Geometric Probability}, Cambridge University Press, 1997.


\bibitem{Mil} J.W.  Milnor: \href{http://www.jstor.org/stable/1969467}{\sl  On the total curvature of knots},  Ann. Math., {\bf 52}(1950), 248-257.


\bibitem{Mor} J.M. Morvan: {\sl  Generalized Curvatures}, Springer Verlag, 2008.

\bibitem{LC} L. Nicolaescu \href{http://www.nd.edu/~lnicolae/con-short.pdf}{\sl On the Normal Cycles of Subanalytic Sets}, Ann. Glob. Anal. Geom.,  {\bf 39}(2011), 427-454. 


\bibitem{Pig} R. Pignoni: {\sl Density and stability of Morse functions on a stratified  space}, Ann. Scuola Norm. Sup. Pisa, series IV, {\bf 6}(1979), 593-608.


\bibitem{Row} B. Rowekamp: {\sl Planar pixelations and image reconstruction},   \href{http://front.math.ucdavis.edu/1105.2831}{\textsf{arXiv: 1105.2831}}

\bibitem{Dr} L. van den Dries: {\sl Tame Topology and $o$-minimal Structures}, London Math. Soc. Lectures Notes Series, vol. 248, Cambridge University Press,  1998.


\end{thebibliography}
\end{document}